\documentclass[12pt]{article}

\usepackage{amsmath}
\usepackage{amssymb}
\usepackage{amsfonts}
\usepackage{amsthm}
\usepackage{enumerate}

\usepackage{graphicx}

\usepackage[margin=1.0in]{geometry}

\newtheorem{theorem}{Theorem}[section]
\newtheorem{proposition}[theorem]{Proposition}
\newtheorem{lemma}[theorem]{Lemma}
\newtheorem{definition}[theorem]{Definition}
\newtheorem{corollary}[theorem]{Corollary}

\newtheorem{remark}{Remark}[section]

\newcommand{\important}[1]{\textbf{#1}}
\newcommand{\allEvery}[1]{\emph{#1}}

\newtheorem{thm}{Theorem}

\theoremstyle{definition}

\theoremstyle{remark}

\newtheorem{cor}[thm]{Corollary}

\newcommand{\norm}[1]{\left\Vert#1\right\Vert}

\newcommand{\RR}{\mathbb{R}}

\newcommand{\ZZ}{\mathbb{Z}}
\newcommand{\vis}{\operatorname{vis}}
\newcommand{\dist}{\operatorname{dist}}

\newcommand{\mass}{\operatorname{mass}}
\newcommand{\supp}{\operatorname{supp}}
\newcommand{\diam}{\operatorname{diam}}

\newcommand{\conv}{\operatorname{conv}}

\numberwithin{equation}{section}

\newcommand{\rr}{{\mathbb R}}
\newcommand{\zz}{{\mathbb Z}}
\newcommand{\nn}{{\mathbb N}}

\newcommand{\Fav}{{\rm Fav}}
\newcommand{\essinf}{{\rm essinf}}

\newcommand{\sss}{\mathbb{S}}

\begin{document}

\title{Quantitative visibility estimates for unrectifiable sets in the plane}
\author{M. Bond, I. {\L}aba, J. Zahl}
\date{July 8, 2014}
\maketitle

\begin{abstract}
The ``visibility'' of a planar set $S$ from a point $a$ is defined as the normalized size of the radial projection of $S$ from $a$ to the unit circle centered at $a$. Simon and Solomyak \cite{SS} proved that unrectifiable self-similar one-sets are invisible from every point in the plane. We quantify this by giving an upper bound on the visibility of $\delta$-neighbourhoods of such sets. We also prove lower bounds on the visibility of $\delta$-neighborhoods of more general sets, based in part on Bourgain's discretized sum-product estimates in \cite{Bourgain1}.
\end{abstract}

\section{Introduction}\label{introductionSection}

Given $a\in\rr^2$, we define the radial projection $P_a:\rr^2\setminus a\to\sss^1$ by
\begin{equation*}
P_a(x):=\frac{x-a}{|x-a|}.
\end{equation*}
The \important{visibility} of a measurable set $S\subset\rr^2$ from $a$ is
\begin{equation*}
\vis(a;S):=\frac1{2\pi}|P_a(S)|,
\end{equation*}
the normalized measure of the set of angles at which $S$ is visible from the \important{vantage point} $a$. Informally,
$\vis(a;S)$ is the proportion of the ``field of vision'' $S$ takes up for an observer situated at $a$.

Suppose that $S\subset\rr^2$ is a one-set, that is, a Borel set whose 1-dimensional Hausdorff measure is positive and finite. Suppose furthermore that $S$ is purely unrectifiable. Marstrand \cite[Sections 8 and 9]{Marstrand} proved that $P_a(S)$ has 1-dimensional Lebesgue measure 0 for all $a\in\rr^2\setminus X$, where the exceptional set $X$ has Hausdorff dimension at most 1, and demonstrated by means of an example that the exceptional set can indeed be 1-dimensional. (See also \cite{mattila81, cs, cs2} for further results on possible sets of vantage points from which a purely unrectifiable one-set can be visible.) In the converse direction,  it follows from Marstrand's projection theorem (\cite{Marstrand}, Theorem II) via projective transformations (cf. Section \ref{sec-projective}) that $P_a(S)$ has Hausdorff dimension 1 for Lebesgue-almost all $a\in\rr^2$. Simple examples (see Section \ref{intro-lower}) show that it is in fact possible for $P_a(S)$ to have dimension less than 1.

However, if $\mathcal J$ is a self-similar set, then stronger statements hold.

\begin{definition}\label{defnOfSelfSim}
(a) A set $\mathcal{J}\subset\rr^2$ is \important{self-similar} if it satisfies the condition
\begin{equation}\label{defnOfJSelfSim}
\mathcal{J} = \bigcup_{i=1}^s T_i(\mathcal J),
\end{equation}
where each map $T_i$ is of the form
\begin{equation}\label{e-similitude}
T_i(x) = \lambda_i \mathcal{O}_i x + z_i.
\end{equation}
Here, $0<\lambda_i<1$, and $\mathcal{O}_i$ is an orthogonal transformation.

(b) We will furthermore say that $\mathcal{J}$ satisfies the \important{Open Set Condition} if there exists an open set $O\subset\rr^2$ such that $\bigcup_{i=1}^s T_i(O)\subset O$ and the sets $T_i(O)$ are disjoint.

\end{definition}

We will refer to the maps in (\ref{e-similitude}) as \important{similitudes}. If $\mathcal{O}_i=I$, then the maps are called \important{homotheties}.

It is well known (\cite{Hutchinson}; see also \cite[Section 8.3]{falconer-book}) that, given the mappings $T_i$ as in (\ref{e-similitude}), there is a unique non-empty compact set $\mathcal{J}$ obeying (\ref{defnOfJSelfSim}).
Assuming the Open Set Condition, the Hausdorff dimension of $\mathcal{J}$ is equal to its \important{similarity dimension}, i.e.  the unique number $\alpha$ such that $\sum_{i=1}^s |\lambda_i|^\alpha=1$. Moreover, $\mathcal{J}$ has positive and finite $\alpha$-dimensional Hausdorff measure (see e.g. \cite[Theorem 8.6]{falconer-book}.) If the points $z_i$ are collinear, then $\mathcal{J}$ is a subset of a line. If $\alpha=1$ and $z_i$ are not all collinear, then $\mathcal{J}$ is a purely unrectifiable one-set.

In \cite{SS}, Simon and Solomyak showed that if $\mathcal{J}$ is a self-similar one-set in the plane satisfying the Open Set Condition and not contained in a line,  then $\vis(a;\mathcal{J})=0$ for \allEvery{every} $a\in\rr^2$ (i.e. $S$ is invisible from every vantage point, with no exceptions). On the other hand, it follows from the results of \cite{HS} and \cite{hochman2} that, under the slightly stronger Strong Separation Condition and assuming that all similarities in (\ref{defnOfJSelfSim}) are homotheties, $P_a(S)$ has Hausdorff dimension 1 for  every $a\in\rr^2$ (see Proposition \ref{hochman-prop}).

We will be interested in quantifying the above estimates, in the sense of proving upper and lower bounds on $\vis(a;S_\delta)$ as $\delta\to 0$, where $S_\delta$ is the $\delta$-neighbourhood of an unrectifiable one-set.
In Section \ref{upper}, we quantify the result of \cite{SS} by proving upper bounds on the visibility of small neighborhoods of 1-dimensional self-similar sets.
Conversely, in Section \ref{lower} we prove lower bounds on the visibility, from all vantage points outside of a small exceptional set, of a more general type of finite scale unrectifiable sets.

We note that there has also been interest in the question of estimating the size of those parts of subsets
of $\rr^n$ that are visible from points or affine subspaces in $\rr^n$, see  \cite{AJJRS, FalconerFraser, JJMO, ONeil}. We refer the reader to \cite{mattila-review} for an introduction to these and other related problems.


\subsection{Upper bounds for self-similar sets}
\label{sec-upper-selfsim}

In this part of the paper, we will only consider 1-dimensional self-similar sets with no rotations and with equal contraction ratios, i.e. sets $\mathcal{J}$ satisfying \eqref{defnOfJSelfSim} where for each $i=1,\ldots,s$ we have $\mathcal{O}_i = I$ and $\lambda_i = \frac{1}{s}$. Without loss of generality we can assume that $\diam\mathcal{J}\sim 1$.

Let $\mathcal{J}_0$ be the convex hull of $\mathcal{J}$, and let $\mathcal{J}_{n+1}:=\bigcup_{i=1}^s T_i(\mathcal{J}_n)$, the $n$-th ``partially constructed'' fractal. Then  $\mathcal{J}_n$ can be covered by $O(1)$ copies of a $\delta$-neighbourhood of $\mathcal{J}$ with $\delta=s^{-n}$ and vice versa. It follows that for the purposes of estimating visibility up to a constant, the $\delta=s^{-n}$ neighborhood of $\mathcal{J}$ is equivalent to $\mathcal{J}_n$.

A model example is the ``4-corner Cantor set." Let $\mathcal{K}=\bigcap_{n=1}^\infty \mathcal{K}_n$, where
\begin{equation*}
\begin{split}
\mathcal{K}_n&=K_n+[0,4^{-n}]^2,\\
K_n&=\Big\{(x,y)\in\rr^2:\ x=\sum_{j=1}^n x_j4^{-j},\ y=\sum_{j=1}^n y_j4^{-j},\ x_j,y_j\in\{0,3\}\Big\}.
\end{split}
\end{equation*}
Geometrically, we start with the unit square, divide it into 16 congruent squares of sidelength $1/4$, keep the 4 squares at the corners while discarding the rest, then iterate the procedure inside each of the four surviving squares.
Then $\mathcal{K}_n$ consists of $4^n$ squares of sidelength $\delta=4^{-n}$, and $\mathcal{K}$ is the Cantor set obtained in the limit. The 4-corner set has long been of interest in complex analysis, as an example of a set with positive 1-dimensional length and zero analytic capacity (\cite{garnett}; see also \cite{tolsa} for an overview of this area of research). Projections of $\mathcal{K}$ and $\mathcal{K}_n$ have been studied e.g. in \cite{PS1, NPV, BaV}.

Our upper bounds on the visibility of self-similar sets will be based on a connection between the visibility problem and estimates on \important{Favard length} (defined below). We will exploit this connection both by adapting Favard length methods to the visibility problem and by explicitly bounding quantities arising in visibility estimates by the Favard length of the set.

The linear projection $\pi_\theta:\rr^2\to\rr$ is given by
\begin{equation}\label{pi theta}
\pi_\theta(x,y)=x\cos\theta + y\sin\theta.
\end{equation}
The \important{Favard length} of a set $S$ is the average (with respect to angle) length of its linear projections:
\begin{equation}\label{fav def}
\Fav(S):=   \frac1{\pi}\int_0^\pi |\pi_\theta(S)|\,d\theta  =  \frac1{\pi}\int_0^\pi \int_\rr \chi_{\pi_\theta(S)}(r)\,dr\,d\theta.
\end{equation}
A theorem of Besicovitch \cite{besic} shows that if $S$ is an unrectifiable one-set, then $|\pi_\theta(S)|=0$ for Lebesgue almost all $\theta$. In particular, $\Fav(S)=0$. It follows that
\begin{equation}\label{fav-zero}
\lim_{\delta\to 0}\Fav(S_\delta)=0.
\end{equation}
However, there may exist exceptional directions $\theta$ for which $|\pi_\theta(S)|>0$: this happens e.g. for $\mathcal{K}_n$ and $\theta=\tan^{-1}(1/2)$. It is in fact possible to have a dense set of such directions \cite{Ke}.

In general, little can be said about the rate of decay of $\Fav(S_\delta)$ as $\delta\to 0$. However, in the case of self-similar sets, effective upper bounds were proved recently in a series of papers starting with \cite{NPV} and continuing in \cite{LZ, BV1, BV3, BLV}. The current state of knowledge may be summarized as follows.

\begin{theorem}\label{thm-upper}
Let $\mathcal{J}$ be a 1-dimensional self-similar set defined by homotheties with equal contraction ratios. Then:
\begin{enumerate}[(i)]
\item If $s\leq 4$, we have $\Fav(\mathcal{J}_n)\leq C n^{-p}$ for some $p>0$ \cite{NPV, BV1, BLV}.
\item The same estimate holds if $\mathcal{J}$ is a self-similar product set which is not a line segment, the similarity centers $z_i$ are rational, and $\pi_{\theta_0}(\mathcal{J})>0$ for some $\theta_0$ \cite{LZ}.
\item If $\mathcal{J}$ is a self-similar product set such that the similarity centers $z_i$ form a product set $A\times B$ with $A,B$ rational and $2\leq |A|,|B|\leq 6$, then $\Fav(\mathcal{J}_n)\leq C n^{-p/\log\log n}$ \cite{BLV}.
\item For general self similar sets defined by homotheties with equal contraction ratios, we have $\Fav(\mathcal{J}_n)\leq C e^{-c\sqrt{\log n}}$ \cite{BV3}.
\end{enumerate}
\end{theorem}

All constants and exponents above depend on the set $\mathcal{J}$. In some additional cases, the assumptions may be weakened and/or the results improved; see \cite{BLV} for details. We also note that a quantitative bound for a special class of self-similar sets with rotations is given in \cite{eroglu}.

Lower bounds on $\Fav(\mathcal{J}_n)$ are much easier to prove. A result of Mattila \cite{Mattila} implies the lower bound
\begin{equation}\label{favard-lower}
\Fav(\mathcal{J}_n)\geq Cn^{-1}
\end{equation}
for general 1-dimensional self-similar sets with equal contraction ratios (allowing rotations).
In \cite{BaV}, this was improved to $(C\log n)/n$ for the 4-corner set $\mathcal{K}_n$.

By interchanging the order of integration, $\Fav(S_\delta)$ may be interpreted as the average value of $|P_a(S_\delta)|$ with respect to $a$ on an appropriate curve in $\rr^2$ (see Proposition \ref{average-vis}). In particular, the Favard bounds just mentioned provide bounds on the averages of $P_a(\mathcal{J}_n)$.

Our theorem provides a pointwise bound quantifying the result of \cite{SS}.

\begin{theorem}\label{fav-upper}
Let $\mathcal{J}$ be self-similar set satisfying the Open Set Condition, whose similitudes have no rotations and have equal contraction ratios. Then for all $a\notin \mathcal{J}$,
\begin{equation*}
\vis(a;\mathcal{J}_n)\leq C_1\sqrt{\Fav(\mathcal{J}_{C_2\log n})}.
\end{equation*}
(The constants are allowed to depend on $a$ and $\mathcal{J}$. It will be clear from the proof that if $\mathcal{J}$ is given, and if $a$ ranges over a fixed compact set disjoint from $\mathcal{J}$, then the constants may be chosen uniform for all such $a$.)
\end{theorem}

The proof is given in Section \ref{upper}. It follows the same rough outline as in \cite{SS}, but we use the methods from the Favard length papers mentioned above to make our estimates effective.

Theorem \ref{fav-upper} should be used in conjunction with Theorem \ref{thm-upper}. For example, for the 4-corner set, Theorem \ref{fav-upper} together with Theorem \ref{thm-upper}(i) implies that
\begin{equation}\label{fav-upper-4c}
\vis(a;\mathcal{K}_n)\leq C (\log n)^{-p/2},
\end{equation}
where $p$ is the same as in Theorem \ref{thm-upper}(i) (in this specific case, by the result of \cite{NPV} we can take any $p<1/6$, with the constant $C=C(p)$ depending on $p$). We also note that the main result of \cite{SS} is more general, allowing 1-dimensional self-similar sets with rotations and not necessarily equal contraction ratios.

Theorem \ref{fav-upper}, as well as the result of \cite{SS}, demonstrate that for self-similar sets, radial projections are ``better behaved" than linear projections. There are many unrectifiable 1-dimensional self-similar sets (e.g. the 4-corner set or the Sierpi\'nski gasket) which project linearly to sets of positive Lebesgue measure in certain directions, so that any results such as \eqref{fav-zero} or Theorem \ref{thm-upper} can only hold in the sense of averages. On the other hand,
\cite{SS} shows that the visibility of the square Cantor set is $0$ from \textit{every} vantage point, and our Theorem \ref{fav-upper} quantifies this. Heuristically, the reason is that radial projections of self-similar sets (even if only from one point) already involve averaging over directions. A similar principle is present in the proof of the lower visibility bound in Proposition \ref{hochman-prop}.


\subsection{Lower bounds on visibility}
\label{intro-lower}

Our next result shows that neighborhoods of discrete unrectifiable sets satisfy visibility lower bounds away from a small exceptional set of vantage points. We will first need several definitions. The following definition is similar to the notion of a $(\delta,\alpha)_2$--set from \cite{Katz}.

\begin{definition}\label{discreteAlphaDim}
Let $\delta>0,0<\alpha\leq 1$ and $C>0$. We say that a set $\mathcal{A}\subset\RR^2$ is \important{an $(\alpha,C,\delta)$--set that is unconcentrated on lines} if the following conditions hold:
\begin{itemize}
 \item $\mathcal{A}$ is a non-empty union of closed $\delta$-balls with at most $C$-fold overlap (i.e. any $x\in \rr^2$ belongs to at most $C$ balls of $\mathcal{A}$).
 \item $C^{-1}\delta^{2-\alpha}\leq |\mathcal{A}|\leq C\delta^{2-\alpha}$.
 \item For every ball $B$ of radius $r$, we have the bound
 \begin{equation}\label{unconcentratedBallEstimate}
 |\mathcal{A}\cap B|\leq C r^\alpha|\mathcal{A}|.
 \end{equation}
 \item For every line $\ell$, we have the bound
  \begin{equation}\label{unconcentratedLineEstimate}
 |\mathcal{A}\cap \ell^{1/C}|\leq |\mathcal{A}|/10,
 \end{equation}
 where $\ell^\rho$ is the $\rho$--neighborhood of $\ell$.
\end{itemize}
\end{definition}

If $\alpha=1$, we will also consider a slightly more specialized type of sets.
\begin{definition}\label{def-sets}
 For $0<\kappa \leq 1/2,$ and $C$ large, we say that $\mathcal{A}$ is a $(\kappa,C,\delta)$--\important{unrectifiable one-set} if $\mathcal{A}$ is a $(1,C,\delta)$--set that is unconcentrated on lines, and if for every rectangle $R$ with dimensions $r_1\leq r_2$, we have
\begin{equation}\label{defnOfOneDimUnrecEqn}
 |\mathcal{A} \cap R |\leq Cr_1^{\kappa}\,r_2^{1-\kappa} |\mathcal{A}|.
\end{equation}
\end{definition}

Note that if $\alpha=1$, then \eqref{defnOfOneDimUnrecEqn} implies \eqref{unconcentratedLineEstimate}, provided that $\mathcal{A}$ is contained in a compact set $K$ (the constant appearing in \eqref{unconcentratedLineEstimate} may depend on $\kappa,C$, and the diameter of $K$).

In applications, the specific values of the constants $\kappa$ and $C$ will not be important. We will also say that $\mathcal{A}$ is equivalent to a $(\alpha,C,\delta)$--set that is unconcentrated on lines if there are $(\alpha,C_i,\delta_i)$--sets $\mathcal{A}_1$, $\mathcal{A}_2$ that are unconcentrated on lines, with $C_i\sim C$ and $\delta_i\sim \delta$ (the $\sim$ notation is explained below), such that $\mathcal{A}_1\subset \mathcal{A}\subset\mathcal{A}_2$. We say that $\mathcal{A}$ is equivalent to a $(\kappa,C,\delta)$--unrectifiable one-set if an analogous property holds. This happens for example for $\mathcal{K}_n$, which obeys all of the above conditions except that it is a union of disjoint squares instead of balls.

If $\mathcal A$ is equivalent to a $(\alpha,C,\delta)$--set that is unconcentrated on lines, we will sometimes abuse the terminology and say simply that $\mathcal{A}$ is a $(\alpha,C,\delta)$--set that is unconcentrated on lines, since for our purposes the distinction is not important. We will adopt a similar convention for sets $\mathcal{A}$ that are equivalent (in the same sense) to $(\kappa,C,\delta)$--unrectifiable one-sets.

\medskip

\noindent{\bf Example 1: Self similar sets.} We prove in Theorem \ref{discrete-unrect}
that if $\mathcal{J}$ is a $\alpha$-dimensional self-similar set (in the sense of Definition \ref{defnOfSelfSim}) with $0<\alpha\leq 1$, satisfying the Open Set Condition and not contained in a line, then its $\delta$-neighbourhood $\mathcal{J}^\delta$ is equivalent to a a $(\alpha,C,\delta)$--set that is unconcentrated on lines, with $C$ independent of $\delta$. Moreover, if $\alpha=1$, then $\mathcal{J}^\delta$ is equivalent to a
$(\kappa,C,\delta)$--unrectifiable one-set for some $C$ and some $\kappa>0$ independent of $\delta$.
It is easy to see from the proof that the same argument extends to modified Cantor constructions that have roughly the same ``distribution of mass" but no exact self-similarities, for example the randomized 4-corner set of \cite{PS1}.

\medskip

\noindent{\bf Example 2: Diffeomorphic images of self similar sets.} In Corollary \ref{diffeoCor} (proved in Section \ref{defnOfFSection} below), we show that diffeomorphic images of $(\kappa,C,\delta)$--unrectifiable one-sets
are equivalent to $(\kappa/2,C',\delta')$--unrectifiable one-sets with $C'\sim C$ and $\delta'\sim\delta$.
In particular, diffeomorphic images of self-similar one-sets provide a rich class of examples.

Similarly in Section \ref{unrectA} we show that if $\alpha<1$, and if $\mathcal{J}$ is an $\alpha$--dimensional self-similar set satisfying the Open Set Condition and $\mathcal{J}$ is not a subset of a line, then diffeomorphic images of $\mathcal{J}^\delta$ are $(\alpha,C,\delta')$--sets that are unconcentrated on lines.
\medskip

Our result is as follows.
\begin{theorem}\label{boundOnHighMultPtsThm}
\textbf{(A):} Let $0<\alpha\leq 1$, and let $U\subset\RR^2$ be a compact set. Let $\mathcal{A}\subset[0,1]^2$ be a $(\alpha,C_0,\delta)$--set that is unconcentrated on lines. For $\lambda \in (0,1],$ we have
\begin{equation}\label{boundOnHighMultPtsCts}
|\{a\in U \colon \vis(a;\mathcal{A}) < \lambda \}|  \leq C_1|\log\delta|^{C_2} \delta^{-2+2\alpha}\lambda^{2}.
\end{equation}
\textbf{(B):} Let $0<\alpha\leq 1$, and let $U\subset\RR^2$ be a compact set. Let $\mathcal{A}\subset[0,1]^2$ be either a $(\alpha,C_0,\delta)$--set that is unconcentrated on lines (for $\alpha<1$) or a $(\kappa,C_0,\delta)$--unrectifiable one-set (if $\alpha=1$). Then there exist constants $\epsilon_0,\epsilon_1>0$ (depending on $\alpha$ and/or $\kappa$) such that for all $\lambda < \delta^{\alpha/2-\epsilon_0},$ we have
\begin{equation}\label{boundOnHighMultPtsCtsEps}
|\{a\in U \colon \vis(a;\mathcal{A}) < \lambda \}|\leq C_3|\log\delta|^{C_4} \delta^{-2+2\alpha}\lambda^{2+\epsilon_1}.
\end{equation}
The constants in the above inequalities depend on $U,$ $\kappa,$ $\alpha$, and $C_0$, but not on $\lambda$ or $\delta$.
\end{theorem}

Theorem \ref{boundOnHighMultPtsThm} is proved in Section \ref{lower}. The estimate \eqref{boundOnHighMultPtsCts}  is based on $L^2$ estimates in incidence geometry. The improvement in \eqref{boundOnHighMultPtsCtsEps} relies on Bourgain's discretized Marstrand projection theorem \cite{Bourgain1}.

Theorem \ref{boundOnHighMultPtsThm} is best understood in the context of specific examples. Let $\mathcal{K}$ be the the 4-corner Cantor set defined in Section \ref{sec-upper-selfsim}.
Let $\mathcal{K}^\prime$ be a 4-corner set in polar coordinates, i.e. the image of $\mathcal{K}$ under the mapping
\begin{equation*}
\phi:\ (x,y)\to ((x+1)\cos \pi y,\ (x+1)\sin \pi y).
\end{equation*}
Since $\phi$ is a diffeomorphism on a neighbourhood of $[0,1]^2$, by Corollary \ref{diffeoCor} (proved in Section \ref{defnOfFSection} below) we have that $\mathcal{K}^\prime_n:=\phi(\mathcal{K}_n)$ is a $(\kappa,C,4^{-n})$--unrectifiable one-set.

By Proposition \ref{hochman-prop}, for every $a\in \rr^2$, $P_a(\mathcal{K})$ has Hausdorff dimension 1. Since Hausdorff dimension provides a lower bound for the box dimension, we have
\begin{equation*}
\vis(a,\mathcal{K}_n)\geq C(a,\epsilon) 4^{-n\epsilon}
\end{equation*}
for every $a\in\rr^2$ and $\epsilon>0$. Pointwise, this is much stronger than Theorem \ref{boundOnHighMultPtsThm}, except for the uniformity in $a$.

Consider now $\mathcal{K}'$. Since $\mathcal{K}'$ is not self-similar, Proposition \ref{hochman-prop} does not apply. Indeed, the conclusion of Proposition \ref{hochman-prop} does not hold for $\mathcal{K}'$, because we have an exceptional point at the origin from which $\mathcal{K}'$ is visible in a set of directions of dimension $1/2$. At this point we have
\begin{equation*}
\vis(0,\mathcal{K}'_n)= 2^{-n}.
\end{equation*}
 It is possible for the set of such exceptional points to be infinite. Indeed, $\mathcal{K}$ has a dense set of directions $\theta$ such that $\pi_\theta(\mathcal{K})$ has Hausdorff and box dimension less than 1, corresponding to exact overlaps between two or more projected squares at some stage of the iteration. Hence, if we let $\mathcal{K}''$ be the image of $\mathcal{K}$ under a projective transformation which maps the ``line at infinity" to a line $\ell_0$ in the plane (cf. Section \ref{sec-projective}), then there is a dense countable set of points on $\ell_0$ from
which $\mathcal{K}''$ is visible in a set of directions of dimension less than 1. For such points, we have
\begin{equation*}
\vis(a,\mathcal{K}''_n)\leq C(a) 4^{-\beta n},\ \beta=\beta(a)>0,
\end{equation*}
with both the constant and the exponent depending on $a$. Theorem \ref{boundOnHighMultPtsThm} gives an upper bound on the measure of the set of such points if $C$ and $\alpha$ are given. (See also Lemma \ref{notTooManyPtsOnLineThm}, where we estimate the measure of the set of exceptional points on a given line.) It seems difficult to determine the actual size of the exceptional set on finite scales, and it is possible that this set could in fact be much smaller than Theorem  \ref{boundOnHighMultPtsThm} allows. On the other hand, improving the estimate in Lemma \ref{notTooManyPtsOnLineThm} cannot be easy, since it would be equivalent (via the machinery of Section \ref{lower}) to improving Bourgain's discretized sum-product theorem.

Theorem \ref{boundOnHighMultPtsThm} can fail in the absence of the unrectifiability condition. If $\mathcal{A}$ is a line segment, say $\mathcal{A}=[-2,-1]\times\{0\}$ and $U=[0,1]\times[-1,1]$, then the set $\{a\in U \colon \vis(a;\mathcal{A}) < \lambda \}$ is an angular segment of width about $\lambda$ and has area $\sim \lambda$, which for small $\lambda$ is worse than the bound $\lambda^2$ in Theorem \ref{boundOnHighMultPtsThm} (A).
On the other hand, if we consider the visibility of such sets from a set $U'$ of vantage points such that $\mathcal{A}$ is unconcentrated on any lines that intersect both $\mathcal{A}$ and $U'$ (e.g. $U'=[0,1]\times [1,2]$ in the above example), then the same result applies with the same proof.

\subsection{Acknowledgements}
The authors would like to thank Michael Hochman for permission to include the argument in Section 2.4. The second author would also like to thank Michael Hochman, Pablo Shmerkin and Boris Solomyak for helpful conversations. We are grateful to the anonymous referee for many comments that helped improve this paper.

The first two authors were supported by the NSERC Discovery Grant RGPIN/229818-2012. The first author is an NSF Postdoctoral Fellow. The third author was supported in part by the Department of Defense through the National Defense Science \& Engineering Graduate Fellowship (NDSEG) Program.

\section{Warm-up results}\label{general}

\subsection{Notation}

Throughout this paper, we will work with a small parameter $\delta>0,$ and we will study the behaviour of various quantities as $\delta\to 0$. All constants and exponents will be independent of $\delta$ unless specified otherwise. We will use $A\lesssim B$ or $A = O(B)$ to mean that $A<CB$ for some absolute constant $C$ which may vary for each instance of the $\lesssim$ notation, but remains independent of $\delta$.
We will also use $A\sim B$ to mean that $A\lesssim B$ and $B\lesssim A$.
We will write $A \lessapprox B$ if $A\lesssim |\log\delta|^M B$, where again $M>0$ may vary from line to line, but remains independent of $\delta$. We will say $A\approx B$ if $A\lessapprox B$ and $B\lessapprox A$.

In the particular context of self-similar sets with uniform contraction ratios, we will have $\delta=s^{-n}$, where $s=O(1)$ is fixed and $n$ is large. Thus for example $A\lesssim B$ means that $A\leq CB$ for some $C$ independent of $n$, and  $A\lessapprox B$ means that $A\lesssim n^M B$, for some $M>0$.

We will use $|S|$ to denote the 1- or 2-dimensional Lebesgue measure of a set $S$, or the cardinality of $S$, depending on context. The $\alpha$-dimensional Hausdorff measure will be denoted by $H^\alpha$. We will write $B(a,r)=\{x\in\rr^2:\ |x-a|\leq r\}$. We also use $\chi_S$ to denote the characteristic function of $S$, and $S^\delta=\bigcup_{x\in S}B(x,\delta)$ for the $\delta$-neighbourhood of $S$.

We will frequently deal with subsets of $S^1$, which we will identify with $[0,2\pi)$. Under this identification, the interval $(a,b)$ will correspond to the circular arc $\{(\cos\theta,\sin\theta)\colon a<\theta<b \}.$ Note that this is well defined even if $a$ or $b$ lies outside the interval $[0,2\pi)$, so sometimes we will allow this to occur. Note also that under this identification, if $a\in[0,2\pi)$ then $a$ and $a+\pi$ are antipodal.

If $\mu$ is a measure on $X$ and $f:X\to Y,$ we define the {pushforward} measure $f \mu$ by $f \mu(E)=\mu(f^{-1}(E))$. If $\nu$ is another measure on $X$, we write $\mu\ll\nu$ to mean $\mu$ {is absolutely continuous with respect to} $\nu$.

\subsection{Visibility and Favard length}

We first note that up to constants, the Favard length of a set can be interpreted as its average visibility from a suitably chosen set of vantage points. For example, we have the following.

\begin{proposition}\label{average-vis}
Suppose $S$ is contained in a right triangle $T'$ with corners $(0,0)$, $(0,R)$, $(R,0)$. Let $I_V$ be the line segment from $(-R,-R)$ to $(-R,2R)$, and let $I_H$ be the line segment from $(-R,-R)$ to $(2R,-R)$. Then
\begin{equation}\label{h1}
\Fav(S)\sim\int_{I_V\cup I_H} \vis(a;S)\, da,
\end{equation}
where the integral in $a$ is taken with respect to the one-dimensional Lebesgue measure.
\end{proposition}

\begin{proof}
We have
\begin{equation}\label{favard-triangle}
\begin{split}
\int_{I_V\cup I_H} \vis(a;S)\, da
&=\int_{I_V\cup I_H} \int_{-\pi}^{\pi} \chi_{P_{a}(S)}(\theta)d\theta da\\
&=\int_{-\pi}^{\pi}|X_\theta| d\theta,
\end{split}
\end{equation}
where $X_\theta=\{a\in I_V\cup I_H:\ \theta\in P_a(S)\}$ is the set of points of $I_V\cup I_H$ from which $S$ is visible at angle $\theta$. It suffices to show that
\begin{equation}\label{fav-equi1}
|X_\theta| \leq 10 |\pi_{\theta+\frac{\pi}{2}}(S)| \hbox{ for }-\pi \leq\theta\leq\pi,
\end{equation}
\begin{equation}\label{fav-equi2}
|X_\theta| \geq |\pi_{\theta+\frac{\pi}{2}}(S)| \hbox{ for }-\frac{\pi}{4}\leq\theta \leq \frac{3\pi}{4}.
\end{equation}
Indeed, the full range of angles at which $S$ is visible from $I_V\cup I_H$  is $-\tan^{-1}2\leq\theta\leq \pi- \tan^{-1}\frac{1}{2}$; for other $\theta$, we have $X_\theta=\emptyset$.
By elementary geometry, we have
\begin{equation}\label{fav-equi3}
|\pi_{\theta+\frac{\pi}{2}}(S)|\geq |X_\theta\cap I_V|\cos\theta +|X_\theta\cap I_H|\sin\theta,
\end{equation}
and moreover, if $-\frac{\pi}{4}\leq\theta\leq \frac{3\pi}{4}$, then the equality in (\ref{fav-equi3}) holds
(since none of $S$ is projected outside of $I_V\cup I_H$).
This immediately implies (\ref{fav-equi2}), since $\sin\theta$ and $\cos\theta$ are bounded by 1.
Furthermore, if $X_\theta\cap I_H$ is non-empty, we must have $ \tan^{-1}\frac{1}{2}\leq\theta\leq \pi- \tan^{-1}\frac{1}{2}$, and in that range of $\theta$ we have $\sin\theta\geq
\sin( \tan^{-1}\frac{1}{2})>\frac{1}{10}$. Similarly, if $X_\theta\cap I_V\neq\emptyset$, we must have $-\tan^{-1}2\leq\theta\leq \tan^{-1}2$, and in that range
we have $\cos\theta>\frac{1}{10}$. This together with (\ref{fav-equi3}) implies (\ref{fav-equi1}).
\end{proof}

The key property of the line segments $I_V$ and $I_H$ is that for every point $x\in S$ and every angle $\theta\in[0,\pi)$, the line passing through $x$ pointing in direction $\theta$ intersects the set $I_V\cup I_H$ at an angle comparable to 1. We could replace the set $I_H\cup I_V$ with other rectifiable curves that have this property: for example, a similar result holds if $S$ is contained in the ball
$B(0,\frac{1}{2})$ and $I_V\cup I_H$ is replaced by the circle $|x|=1$.


\subsection{Energy methods}
For a compact set $S\subset\rr^d$, let $\mathcal{M}(S)$ denote the set of all non-negative Radon probability measures supported on $S$. The (Riesz) $s$-energy of $\mu\in\mathcal{M}(S)$ is given by
\begin{equation*}
I_s(\mu):=\int\int |x-y|^{-s}d\mu(x)\,d\mu(y).
\end{equation*}
We will require the following characterization of the Hausdorff dimension of $S$ (see \cite[Theorems 8.8 and 8.9]{mattila-book} or \cite[Propositions 8.2 and 8.4]{wolff-lectures}):
\begin{align*}
\dim(S)&
=\sup\{s>0:\  \exists\ \mu\in\mathcal{M}(S)\text{ such that }I_s(\mu)<\infty\}\\
&=\sup\{s>0:\  \exists\ \mu\in\mathcal{M}(S)\text{ such that }\mu(B(x,r))\lesssim r^s\text{ for all }x\in\rr^d,r>0\}.
\end{align*}
(By convention, if the sets on the right are empty, we will consider their suprema to be 0; however, the results below are only of interest if $\dim(S)>0$.)

The following is a visibility analogue of a well known result of Kaufman \cite{kaufman}. We will use $\ell_{x,y}$ to denote the line through $x$ and $y$, and $\ell_{x,y}^\rho$ to denote the $\rho$--neighbourhood of $\ell_{x,y}$.

\begin{theorem}\label{kauf-vis}
Let $S\subset\RR^2$ be measurable, and consider a set of vantage points $V\subset\RR^2$ equipped with a measure $\nu\in\mathcal{M}(V)$.
Assume that $\dist (V,S)\gtrsim 1$, and that for some $\beta>0$ we have
\begin{equation}\label{e-unconcentrated}
\nu(\ell_{x,y}^\rho)\lesssim\rho^\beta \hbox{ for all }\rho>0\hbox{ and }x,y\in S,\ x\neq y.
\end{equation}
Then for all $s<\min\{\beta,\dim S\}$ we have
\begin{equation*}
\nu\{a:\dim(P_a(S))< s\}=0.
\end{equation*}
\end{theorem}

The conclusion of the theorem obviously fails if both $S$ and $V$ lie on the same straight line. The assumption (\ref{e-unconcentrated}) excludes pathological cases of this type.
In particular, if $S\subseteq T'$ (the triangle from Proposition \ref{average-vis}), then (\ref{e-unconcentrated}) holds with $\beta=1$ if $\nu$ is the Hausdorff measure on $I_V\times I_H$. Moreover, if $V$ is compact and $\nu = H^\alpha$ for some $\alpha>1$, then (\ref{e-unconcentrated}) holds with $\beta=\alpha-1$.

The proof uses a simple geometric lemma.

\begin{lemma}\label{freddie}
Suppose $x\neq y$ and $|a-x|\sim|a-y|\sim 1$. Then
\begin{equation}\label{PaxPayComp}
|P_a(x)-P_a(y)|\gtrsim |x-y|\dist(a,\ell_{x,y}).
\end{equation}
Here $|P_a(x)-P_a(y)|$ denotes the arc-length of the interval of $S^1$ with endpoints $P_a(x)$ and $P_a(y)$.
\end{lemma}
\begin{proof}
Let $\gamma\in[0,\pi)$ be the angle between the half-lines from $a$ to $x$ and $y$, so that $\gamma=|P_a(x)-P_a(y)|$. We will always assume that $\gamma\leq\pi/10$, since otherwise there is nothing to prove. Let $b$ be the orthogonal projection of $a$ on the line $\ell_{x,y}$; in particular, $|b-a|=\dist(a,\ell_{x,y})$. Let $\gamma_x$ be the angle between the half-lines $xa$ and $xb$, and define $\gamma_y$ similarly. First consider the case where $b$ lies in the interval $xy$, so that $\gamma=\pi- \gamma_x-\gamma_y$. Since $\gamma\leq\pi/10$ and $\gamma_x,\gamma_y\leq\pi/2$, we have $\gamma_x,\gamma_y\geq 2\pi/5.$ We can now bound
\begin{align*}
\gamma&=(\frac{\pi}{2}-\gamma_x)+ (\frac{\pi}{2}-\gamma_y)\gtrsim \tan(\frac{\pi}{2}-\gamma_x)+ \tan(\frac{\pi}{2}-\gamma_y)\\
&=\frac{|b-x|}{|a-b|}+\frac{|b-y|}{|a-b|}=\frac{|x-y|}{|a-b|}\gtrsim |x-y|\,|a-b|,
\end{align*}
which establishes \eqref{PaxPayComp} in this case. The last inequality follows from the observation that $|a-b|\lesssim 1$.

Now suppose that $b$ lies outside of the interval $xy$.  Then (re-labeling $x$ and $y$ if necessary) we have that $\gamma=\gamma_x-\gamma_y$. In particular, the triangle spanned by $a,x,$ and $y$ is obtuse, and the vertex $x$ has the largest angle. Call this angle $\beta$. First consider the case where $\gamma_x>\pi/10$, so $\pi/2\leq\beta\leq 9\pi/10$. By the law of sines, we have
\begin{equation*}
\frac{\gamma}{|x-y|}\geq\frac{\sin\gamma}{|x-y|} =\frac{\sin\beta}{|a-y|}\gtrsim 1\gtrsim |a-b|,
\end{equation*}
which establishes \eqref{PaxPayComp}.

Finally, consider the case where $\gamma_x\leq\pi/10$. Then $\beta\geq 9\pi/10,$ so in particular we have $|x-y|\sim\big| |a-x|-|a-y|\big|$. Since $\frac{\sin u-\sin v}{u-v}\leq 1$ for $u\neq v$, we have

\begin{align*}
\gamma_x-\gamma_y&\geq\sin\gamma_x-\sin\gamma_y=\left|\frac{|a-b|}{|a-x|}-\frac{|a-b|}{|a-y|}\right|\\
&\sim |a-b|\, \big||a-y|-|a-x|\big|\sim |x-y|\,|a-b|
\end{align*}
as claimed.
\end{proof}

\begin{proof}[Proof of Theorem \ref{kauf-vis}.]
Let $s<\min\{\beta,\dim S\}$, and let $\mu\in\mathcal{M}(S)$ such that $I_s(\mu)<\infty$.
It suffices to prove that $I_s(P_{a}\mu)<\infty$ for $\nu$-a.e. $a$. This follows when we prove that $I:=\int I_s(P_{a}\mu)\,d\nu(a)<\infty$. We have

\begin{align*}
I&=\iiint |x-y|^{-s}\,dP_{a}\mu(x)\,dP_{a}\mu(y)\,d\nu(a)\\
&=\iiint |P_a(x)-P_a(y)|^{-s}\,d\mu(x)\,d\mu(y)\,d\nu(a)\\
&\lesssim\iiint (\dist(a,\ell_{x,y}))^{-s}|x-y|^{-s}d\nu(a)\,d\mu(x)\,d\mu(y)\\
&=\iiint_0^\infty \nu(\{a:\dist(a,\ell_{x,y})\leq r^{-1/s}\})\,dr\,|x-y|^{-s}\,d\mu(x)\,d\mu(y).
\end{align*}
On the third line, we used Lemma \ref{freddie}.
By (\ref{e-unconcentrated}),
\begin{align*}
I&\lesssim\iint\big[1+\int_1^\infty r^{-\beta/s}\,dr\big]|x-y|^{-s}\,d\mu(x)\,d\mu(y)\\
&\lesssim I_s(\mu)<\infty.\qedhere
\end{align*}
\end{proof}

The next theorem is an analogue of  \cite[Section 9.10]{mattila-book}, with $\gamma$ equal to the 1-dimensional Lebesgue measure on $I_V\cup I_H$.
It does not seem to generalize well to other vantage sets $V$. We omit the details.

\begin{theorem}
Assume that $S\subset T'$ is compact, with $T'$ as in Proposition \ref{average-vis}. Let $\mu\in\mathcal{M}(S)$. Then
\begin{equation}\label{applepie1}
\int_{I_V\cup I_H} |P_a(S)|^{-1}\, da \lesssim I_1(\mu),
\end{equation}
\begin{equation}\label{applepie2}
I_1(\mu)^{-1}\lesssim \int_{I_V\cup I_H} |P_a(S)|\, da.
\end{equation}
\end{theorem}

Together with Proposition \ref{average-vis}, (\ref{applepie2}) recovers Mattila's lower bound $\Fav(S)\gtrsim I_1(\mu)^{-1}$.
In particular, if $\mu_n$ is the normalized Lebesgue measure on $\mathcal{K}_n$, then a computation similar to that in Lemma \ref{BoundOn2NormOfF} shows that $I_1(\mu_n)\sim n$. It follows that\footnote{The bound \eqref{simpleFavCantorBound} is strengthened to $\text{Fav}(\mathcal{K}_n)\gtrsim\frac{\log n}{n}$ in \cite{BaV}, but we will not need this improvement.}
\begin{equation}\label{simpleFavCantorBound}
\operatorname{Fav}(\mathcal{K}_n)\gtrsim\frac1{n}.
\end{equation}
By Chebyshev's inequality and (\ref{applepie1}), we have that for all $\lambda>0$,
\begin{equation}\label{applepiedeluxe}
\begin{split}
|\{a\in I_V\cup I_H:|P_a({K}_n)|\leq\lambda\}|&\leq\lambda\int_{I_V\cup I_H} |P_a(\mathcal{K}_n)|^{-1}\, da\\
&\lesssim \lambda\ I_1(\mu_n)\\
&\lesssim \lambda n.
\end{split}
\end{equation}
The bound \eqref{applepiedeluxe} should be compared to Lemma \ref{notTooManyPtsOnLineThm}, where under some additional assumptions of $\lambda$ and the set $S$, the RHS of \eqref{applepiedeluxe} is improved to $\lambda^{1+\epsilon}n^C;$ here $\epsilon>0$ is a small constant, and $C$ is a large constant. If $\lambda$ is much smaller than $n^{-1},$ then this is indeed a better bound.


\subsection{Visibility dimension of self-similar sets}

The following argument is due to Michael Hochman and we thank him for permission to include it here. It is very similar to the proof of Theorem 1.7 in \cite{HS2}.

\begin{proposition}\label{hochman-prop}
Let $\mathcal{J}\subseteq\mathbb{R}^{2}$ be a self-similar set satisfying the Strong Separation Condition, and
satisfying \eqref{defnOfJSelfSim} with no rotations (i.e. $\mathcal{O}_i = I$ for each $i=1,\ldots,s$).
Then for any $a\in\mathbb{R}^{2}\setminus\mathcal{J}$ we have
$\dim P_a(\mathcal{J})=\min\{1,\dim \mathcal{J}\}$.
\end{proposition}

The assumption that $a \notin\mathcal{J}$ guarantees that $P_a$ is well defined (and $C^2$ as required below) on all of $\mathcal{J}$. However, if $a \in\mathcal{J}$, we may apply Proposition \ref{hochman-prop} to one of the sets $T_i(\mathcal{J})\subset \mathcal{J}$ from \eqref{defnOfJSelfSim}, and the conclusion follows again.

While the proof itself is short, it relies on major results from \cite{hochman2}, \cite{HS}, and on the machinery developed therein. We present a heuristic argument first, with the rigorous proof to follow.

\begin{proof}[Heuristic proof.]
Let $f:\mathbb{R}^{2}\rightarrow\mathbb{R}$ be differentiable at $b\in\rr^2$, and assume that $\nabla f(b)\neq 0$.
Then for $z=(x,y)$ in a small neighbourhood of $b$, we may approximate $f(z)$ by $f(b)+(Df)_b(z-b)$, where
$$(Df)_b(z)=\nabla f(b)\cdot z = |\nabla f(b)|\pi_\theta(z)$$
is a linear mapping from $\rr^2$ to $\rr$, and $\theta=\theta(f;b)$ is the angle that $\nabla f(b)$ makes with the positive $x$-axis.

Our intended application is to the visibility problem.
Let $a \notin\mathcal{J}$; without loss of generality, we may assume that $a=0$. Let $f=P_0:\rr^2\setminus \{0\}\to\rr/2\pi\zz$ (we identify the latter with $S^1$). For $b=(r\cos\phi,r\sin\phi)$ with $r>0$, we have $\theta(P_0;b)=\phi+\pi/2$, so that $\pi_{\theta(P_0;b)}$ is the orthogonal projection to a line perpendicular to the line through $0$ and $b$.

The idea is to ``linearize" the problem: near each $b\in\mathcal{J}$, we may approximate the radial projection $P_0$ by the linear projection $\pi_{\theta(P_0;b)}$. By self-similarity, arbitrarily small neighbourhoods of every $b\in\mathcal{J}$ contain complete affine copies of $\mathcal{J}$. Therefore the dimension of $P_0(\mathcal{J})$ is bounded from below by the supremum of the dimensions of the corresponding linear projections of such copies. (This is a vast oversimplification; the rigorous version of this argument is given by Theorem 1.13 of \cite{HS}.)

The theorem will now follow if we can find a point $b\in\mathcal{J}$ such that $\dim \pi_{\theta(P_0;b)}(\mathcal{J})=\min (1,\dim \mathcal{J})$.
By Theorem 1.8 of \cite{hochman2}, we have
\begin{equation}\label{h-e1}
\dim \pi_\theta(\mathcal{J})=\min (1,\dim \mathcal{J})
\end{equation}
for all $\theta\notin X$, where $X\subset S^1$ is an exceptional set of dimension 0.
Suppose that we know {\it a priori} that $P_0(\mathcal{J})$ has positive dimension. Then the set $\Omega:=\{\theta(P_0;b):\ b\in\mathcal{J}\}$ also has positive dimension, in particular it cannot be entirely contained in $X$. It follows that (\ref{h-e1}) holds for some $\theta\in\Omega$, hence the conclusion follows as claimed.

To complete the argument, we need to bootstrap. We have to prove that $\dim\Omega>0$. This requires another application of \cite[Theorem 1.13]{HS}, this time linearizing the mapping $g(b)=\theta(P_0;b)$. With notation as above, we have $\nabla g(b)=r^{-1}(-\sin\phi,\cos\phi)$, so that $\theta(g;b)=\phi+\pi/2=\theta(f;b)$. However, now we only need to prove that the dimension of $g(\mathcal{J})$ is positive, not necessarily maximal, so that it suffices to show that there is an $\alpha>0$ such that $\dim \pi_\theta(\mathcal{J})>\alpha$ for all $\theta$. But this is easy to prove, see Lemma \ref{easy}.
\end{proof}

 We now present the rigorous argument for more general mappings.

\begin{proposition}\label{hochman-prop-b}
Let $\mathcal{J}\subseteq\mathbb{R}^{2}$ be a self-similar set defined by homotheties
(i.e. $\mathcal{O}_i = I$ for each $i=1,\ldots,s$), and satisfying the Strong Separation Condition. Let
$\mu$ be the self-similar measure on $\mathcal{J}$ achieving the Hausdorff
dimension.
Suppose that $f:\mathbb{R}^{2}\rightarrow\mathbb{R}$ is a $C^{2}$
map such that the mapping $g:\rr^2\to\rr$ given by $g(b)=\theta(f;b)$ is well defined and obeys $\nabla g(b)\neq0$ except for a $\mu$-null set of points.
Then
\begin{equation}\label{h-e2}
\dim f\mu=\min(1,\dim \mathcal{J})
\end{equation}
\end{proposition}

\begin{proof}[Proof of Proposition \ref{hochman-prop}]
We may assume that $\mathcal{J}\subseteq[0,1]^{2}$.
In this proof only, we will use freely the notation and terminology from \cite{hochman1} and \cite{HS}.

By \cite[Section 4.3]{hochman1}, there is an ergodic CP-distribution
$P$ (see \cite[Section 1.4]{hochman1} for a definition)
such that for a.e. realization $\nu$ of $P$ we have $\mu\ll S\nu$
for some homothety $S$. Let
\begin{equation*}
d(\theta)=\int\dim\pi_{\theta}(\nu)\, dP(\nu).
\end{equation*}
By Theorem 1.22 of \cite{hochman1} (see also \cite[Theorem 1.10]{HS}), for every angle
$\theta$ the following holds: for $P$-a.e. realization of $\nu$ we have that $\pi_{\theta}\nu$ is exact dimensional and $\dim\pi_{\theta}\nu=d(\theta)$, so that $\pi_{\theta}\mu$ is also exact dimensional and $\dim\pi_{\theta}\mu=d(\theta)$.
Then by \cite[Theorem 1.13]{HS}, we have
\begin{equation}
\underline{\dim} f\mu\geq\essinf_{b\in\mathcal{J}}\dim \pi_{\theta(f;b)}\mu,\label{eq:essinf-bound}
\end{equation}
where the essential infimum is taken with respect to $\mu$, and $\underline{\dim}\sigma=\inf\{\dim F\,:\,\sigma(F)>0\}$.

In light of (\ref{h-e1}),  the proof of the proposition reduces now to showing that
\begin{equation*}
\mu(x\in\mathbb{R}^{2}\,:\,\theta(f;x)\in X)=0,
\end{equation*}
or equivalently, that
\begin{equation*}
(g\mu)(X)=0.
\end{equation*}
This will follow if we show that $\underline{\dim} g\mu>0$.
Applying  \cite[Theorem 1.13]{HS} as in (\ref{eq:essinf-bound}) again, but this time to $g$ instead of $f$, we get
\begin{equation*}
\underline{\dim} g\mu\geq\essinf_{b\in\mathcal{J}}\dim\pi_{\theta(g;b)}\mu,
\end{equation*}
which is well defined since $Dg\neq0$ $\mu$-a.e. By Lemma \ref{easy}, we have $\dim\pi_{\theta}\mu\geq \alpha>0$ for all $\theta$.
Therefore $\underline{\dim} g\mu\geq\alpha>0$, as desired.
\end{proof}


\section{Visibility lower bounds}\label{lower}
In this section we will prove Theorem \ref{boundOnHighMultPtsThm}. We will begin with a brief sketch that illustrates the main ideas in the proof.

We begin with Theorem \ref{boundOnHighMultPtsThm}A, which is essentially an incidence result that uses $L^2$/Cauchy-Schwartz type techniques.
Assume that $\mathcal{A}$ is a $(\alpha,C_0,\delta)$--set that is unconcentrated on lines. Let $\mathcal{G}$ be a set of vantage points from which $\mathcal{A}$ has small visibility, in the sense that $\vis(a,\mathcal A)<\lambda$ for $a\in \mathcal{G}$. We may assume that both $\mathcal{G}$ and $\mathcal{A}$ are contained in a ball of radius $\lesssim 1$. We discretize the problem, replacing $\mathcal{G}$ and $\mathcal{A}$ by their maximal $\delta$-separated subsets $G$ and $A$ respectively. Then $|A|\sim \delta^{-2}|\mathcal{A}|\sim \delta^{-\alpha}$, and $|G|\gtrsim\delta^{-2}|\mathcal{G}|$.

Let $a\in G$. The small visibility bound means that $A$ is contained in about $\delta^{-1}\lambda$ rectangles with dimensions about $1\times\delta$ passing through $a$. Let $\mu$ be the typical number of points of $A$ contained in such rectangles; we will also assume that $\mu$ is the same for all points $a\in G$. This can be achieved via dyadic pigeonholing, modulo logarithmic factors that we will ignore in this informal sketch.
The number of such ``rich" rectangles through each $a\in G$ is about $|A|\mu^{-1}=\delta^{-\alpha}\mu^{-1}$. Note that this must be no greater than $\delta^{-1}\lambda$, so that
\begin{equation}\label{sketch-e1}
\mu\gtrapprox \delta^{1-\alpha}\lambda^{-1}.
\end{equation}
Consider the set $\mathcal{T}$ of all triples  $(a,\ell_1,\ell_2)$, where $a\in G$ and  $\ell_1,\ell_2$ are rich rectangles through $a$. The total number of such triples should be about
\begin{equation}\label{sketch-e2}
|\mathcal{T}|\sim |G|(\delta^{-\alpha}\mu^{-1})^2=|G|\delta^{-2\alpha}\mu^{-2}.
\end{equation}
On the other hand,
an $L^2$ argument based on the size and distribution of $A$ shows that the total number of $1\times\delta$ rectangles containing $\mu$ points of $A$ is at most $\delta^{-1-\alpha}\mu^{-2}$
(see (\ref{upperBoundOnH1})).
Assume that any two rich rectangles intersect at an angle $\sim 1$. (This is actually false as stated; instead, we will rely on a ``bilinear" reduction from Section \ref{sec-bilinear}, choosing two subfamilies $\mathcal{H}_1$, $\mathcal{H}_2$ of rich rectangles so that any $\ell_1\in \mathcal{H}_1$ and $\ell_2\in\mathcal{H}_2$ intersect at an angle $\sim 1$.) Then given a pair $\ell_1,\ell_2$ of rich rectangles, there can be only a bounded number of points $a\in G$ contained in their intersection. Thus
\begin{equation}\label{sketch-e3}
|\mathcal{T}|\lessapprox (\delta^{-1-\alpha}\mu^{-2})^2=\delta^{-2-2\alpha}\mu^{-4}.
\end{equation}
Comparing this to (\ref{sketch-e2}), and using also (\ref{sketch-e1}), we get that $|G|\lessapprox \delta^{-2}\mu^{-2}\lessapprox \delta^{-4+2\alpha}\lambda^2$, so that $| \mathcal{G}| \lessapprox\delta^{-4+2\alpha}\lambda^2$
as claimed.

The proof of Theorem \ref{boundOnHighMultPtsThm}B relies on an improvement to (\ref{sketch-e3}). Namely, we will prove that under the assumptions of the theorem, for a $1\times\delta$ rectangle $\ell_1$ we have
\begin{equation}\label{sketch-e4}
|G\cap \ell_1|\lessapprox \delta^{\alpha-2}\lambda^{1+\epsilon_1}
\end{equation}
for some $\epsilon_1>0$. (This is the discretized version of (\ref{boundOnHighMultPtsCtsEpsOneD}).) Recall that each $a\in G$ is contained in at most $\delta^{-1}\lambda$ rich rectangles. Thus, the number of triples $(a,\ell_1,\ell_2)$ in $\mathcal{T}$ with $\ell_1$ fixed is at most $ \delta^{\alpha-2}\lambda^{1+\epsilon_1}\delta^{-1}\lambda= \delta^{\alpha-3}\lambda^{2+\epsilon_1}$. Recalling also the bound $\delta^{-1-\alpha}\mu^{-2}$ on the number of rich rectangles, we can improve (\ref{sketch-e3}) to
$$
|\mathcal{T}|\lessapprox \delta^{-1-\alpha}\mu^{-2} \delta^{\alpha-3}\lambda^{2+\epsilon_1}= \delta^{-4}\mu^{-2} \lambda^{2+\epsilon_1}.
$$
Comparing this to  (\ref{sketch-e2}) as above, we get the desired bound $|G|\lessapprox \delta^{-4+2\alpha}\lambda^{2+\epsilon_1}$.

Suppose for a contradiction that (\ref{sketch-e4}) fails for some rectangle $\ell_1$. Thus $\ell_1$ contains at least $ \delta^{\alpha-2}\lambda^{1+\epsilon_1}$ points
$a\in G$, each of them meeting at least $\delta^{-1}\lambda$ rich rectangles
We may further reduce to the case when the union of these rich rectangles covers $\mathcal{A}$. Abusing notation slightly, we identify the rectangle $\ell_1$ with the line containing its long axis, and apply a projective transformation that sends this line to the line at infinity. Let $\mathcal{A}^\prime$ be the image of $\mathcal{A}$ after this projective transformation, and let $\Theta$ be the image of $G\cap \ell_1$. We conclude that for each $\theta\in\Theta$, the projection of $\mathcal{A}^\prime$ in the direction $\theta$ has size at most $\delta^{1/2-\epsilon_0}=|\mathcal{A}|^{1/2-\epsilon_0}$. Furthermore, the set of directions $\Theta$ does not concentrate too much on small intervals (we refer to this property as being ``well distributed"). However, Bourgain's discretized sum-product theorem does not allow this to happen. This contradiction establishes the theorem.
Note that the condition $\lambda<\delta^{1/2-\epsilon_0}$ is not needed for the above reductions, but it is a key part of Bourgain's theorem.

%
\subsection{Initial reductions and discretization}
We now turn to the proof of Theorem \ref{boundOnHighMultPtsThm}.
We may assume that $U\subset B(0,d)$ for some fixed $d\sim 1$. All constants in the sequel may depend on $d$, but we will not display that dependence. \subsubsection{Discretization of the points}
First, we will need a discretized analogue of $(\alpha,C,\delta)$--sets that are unconcentrated on lines and $(\kappa,C,\delta)$--unrectifiable one-sets.

\begin{definition}
Let $A\subset\RR^2$ be a finite set of points.  We say that $A$ is a \important{discrete $(\alpha,C,\delta)$--set that is unconcentrated on lines} if the following conditions hold:
 \begin{itemize}
 \item $A$ is $\delta$-separated, in the sense that if $a,a'\in A$ and $a\neq a'$, then $|a-a'|\geq\delta$ (in particular, we have $|A\cap B|\leq C$ for any $\delta$--ball $B$).
 \item $C^{-1}\delta^{-\alpha}\leq |A|\leq C\delta^{-\alpha}$.
 \item For every ball $B$ of radius $r\geq\delta$, we have the bound
\begin{equation}\label{unconcentratedBallEstimateDiscrete}
|A\cap B|\leq C r^\alpha|A|.
\end{equation}
\item For every line $\ell$, we have the bound
\begin{equation}\label{unconcentratedLineEstimateDiscrete}
|A\cap \ell^{1/C}|\leq |A|/10.
\end{equation}
\end{itemize}
\end{definition}
\begin{definition}
If $\kappa>0$ and $A$ is a discrete $(1,C,\delta)$--set that is unconcentrated on lines, then we say that $A$ is a \important{discrete $(\kappa,C,\delta)$--unrectifiable one-set} if for every rectangle $R$ of dimensions $\delta\leq r_1\leq r_2$, we have
\begin{equation}\label{defnOfOneDimUnrecEqnDiscrete}
|A \cap R |\leq Cr_1^{\kappa}\,r_2^{1-\kappa} |A|.
\end{equation}
\end{definition}

The following is then clear from the definition.

\begin{lemma}\label{discreteCtsEquivalenceLem}
Let $\mathcal{A}\subset\RR^2$ be a $(\alpha,C,\delta)$--set that is unconcentrated on lines, and let $A$ be a maximal $\delta$--separated subset of $\mathcal{A}$. Then $A$ is a discrete $(\alpha,C^\prime,\delta)$--set that is unconcentrated on lines. Conversely, if $A$ is a $(\alpha,C,\delta)$--set that is unconcentrated on lines, then both $\bigcup_{x\in A} B(x,\delta)$ and $\bigcup_{x\in A} B(x,2\delta)$ are $(\alpha,C^\prime,\delta)$--sets that are unconcentrated on lines. The constant $C^\prime\sim C$ depends only on $C$. An analogous statement holds if $\mathcal{A}$ is a $(\kappa,C,\delta)$--unrectifiable one-set.
\end{lemma}

\subsubsection{Discretization of the lines}
Let $\mathcal{L}_{\delta}$ be a maximal $\delta$--discretized collection of lines that meet the ball $B(0,d)$. For example, we may define
\begin{equation*}
 \mathcal{L}_{\delta} := \{\ell_{k_1,k_2}:\ k_1\in\ZZ\cap[0, \pi\delta^{-1}], \ k_2\in\ZZ\cap[0, d\delta^{-1}]\},
\end{equation*}
where $\ell_{k_1,k_2}$ is the line parallel to the vector $(\cos (k_1\delta),\sin(k_1\delta))$ and passing through the point
$(-k_2\sin(k_1\delta)$, $k_2\cos(k_1\delta))$. Note that $|\mathcal{L}_{\delta}| \sim \delta^{-2}$.
We use $\ell^\rho$ to denote the $\rho$--neighborhood of $\ell$, and $\theta(\ell)$ the direction of $\ell$. By convention, the direction of $\ell$ will always lie in the interval $[0,\pi)$.
\subsubsection{Discretization of visibility}
Let
\begin{equation*}
\vis_{\delta}(a;S) := |\{\ell\in\mathcal{L}_{\delta}\colon a\in\ell^{2\delta}, \ell^{c\delta}\cap S \neq\emptyset \}|.
\end{equation*}

It is then easy to see that if $\mathcal{A}$ is a union of $\delta$--balls, $A$ is a maximal $\delta$--separated subset of $A$, $a,a^\prime\in B(0,d)$ are two points such that $|a-a^\prime|<\delta$, and if $c\sim 1$ is large enough (depending on $d$), then
\begin{equation}\label{VisVisDeltaComparable}
\vis_{\delta}(a^\prime;A) \sim \delta^{-1}\vis(a;\mathcal{A}).
\end{equation}
Note the $\delta^{-1}$ factor, which reflects the fact that $\vis_{\delta}(a;A)$ uses a counting measure that has total mass $\sim \delta^{-1}$. We are also now using lines instead of half-lines; this increases the visibility by at most a factor of 2.

\subsubsection{Separating the vantage points from the set $A$}
For technical reasons, the proof is simpler if every point $a\in U$ has separation $\sim 1$ from $A$. Luckily, we can reduce to this case. Let $r\sim 1$ be small enough so that any ball of radius $2r$ contains no more than half of the mass of $A$; this is possible by (\ref{unconcentratedBallEstimate}). Cover $U$ by $O(1)$ balls $B(a_i,r)$. Increasing $d$ if necessary, we may assume that they are all contained in $B(0,d)$. Then Theorem \ref{boundOnHighMultPtsThm} follows if we can prove the estimates \eqref{boundOnHighMultPtsCts} and \eqref{boundOnHighMultPtsCtsEps} with $U$ replaced by $B(a_i,r)$ and $A$ replaced by $A\setminus B(a_i,2r)$ for each $i$.

Combining the above reductions, we see that it suffices to establish the following theorem.
\begin{theorem}\label{boundOnHighMultPtsBallThm}
\textbf{(A)} Let $0<\alpha\leq 1$ and $d\sim 1$, and let $A\subset[0,1]^2$ be a discrete $(\alpha,C_0,\delta)$--set that is unconcentrated on lines. Then, if $B_0$ is a ball of radius $r\sim 1$, $B_0\subset B(0,d)$, and $\dist(B_0,A)\geq r$, we have that for $\lambda \in (0,1],$
 \begin{equation}\label{boundOnHighMultPtsInBall}
|\{a\in B_0 \colon \vis_\delta(a;A) < \lambda\delta^{-1} \}|\lessapprox \delta^{-2+2\alpha}\lambda^{2}.
\end{equation}
\textbf{(B)} Let $A\subset[0,1]^2$ be either a discrete $(\alpha,C_0,\delta)$--set that is unconcentrated on lines (for $\alpha<1$) or a discrete $(\kappa,C_0,\delta)$--unrectifiable one-set (if $\alpha=1$). Let $d\sim 1$. Then there exist constants $\epsilon_0,\epsilon_1>0$ (depending on $\alpha$ and/or $\kappa$) such that the following holds: If $B_0$ is a ball of radius $r\sim 1$, $B_0\subset B(0,d)$, and $\dist(B_0,A)\geq r$, we have that for for all $\lambda < \delta^{\alpha/2-\epsilon_0},$
\begin{equation}\label{boundOnHighMultPtsCtsEps-b}
|\{a\in B_0 \colon \vis_\delta(a;A) < \lambda\delta^{-1} \}|\leq C_1|\log\delta|^{C_2} \delta^{-2+2\alpha}\lambda^{2+\epsilon_1}.
\end{equation}

The implicit constants may depend on $B_0,$ $\kappa,$ $\alpha$, and $C_0$, but not on $\lambda$ or $\delta$.
\end{theorem}
\begin{remark}
The assumptions of Theorem \ref{boundOnHighMultPtsBallThm}B require that $\lambda < \delta^{\alpha/2-\epsilon_0}$ because this is needed in the proof of Theorem \ref{BourgainsDiscreteSumProdThm} (Bourgain's discretized sum-product theorem), which in turn is the key ingredient of our proof. Specifically, the proof of Theorem \ref{BourgainsDiscreteSumProdThm} relies on the Balog-Szemer\'edi-Gowers theorem to convert the set $A\subset \RR^2$ into a product set $A_1\times A_2$, where $A_1,A_2$ are subsets of $\RR$ that have a certain special structure (they look like a discretized sub-ring of $\RR$). This theorem only works if $\epsilon_0$ is small.

The reason $\epsilon_1$ is small is that Bourgain's theorem only gives us a small gain over the bound we would obtain from more elementary $L^2$ methods.
\end{remark}


%
%
\subsubsection{Some properties of $(\alpha, C,\delta)$-sets}\label{defnOfFSection}

Let $A\subset[0,1]^2$ be a set of points. Motivated by the recent work on the Favard problem (cf. \cite{NPV, BaV}), we define the function $f_{\delta}\colon\mathcal{L}_{\delta} \to \RR$ as follows, with the same $c\sim 1$ as in \eqref{VisVisDeltaComparable}:
\begin{equation}
f_{\delta}(\ell):= |A\cap\ell^{c\delta}|.
\end{equation}

\begin{lemma}\label{BoundOn2NormOfF}
If $A \subset[0,1]^2$ is a discrete $(\alpha,C_0,\delta)$--set that is unconcentrated on lines, then
\begin{equation}
 \norm{f_{\delta}}_{2}^2 \lessapprox \delta^{1-\alpha},
\end{equation}
where
\begin{equation}\label{L2NormOfFndelta}
\norm{f_{\delta}}_{2}^2=\frac{1}{|\mathcal{L}_{\delta}|} \sum_{\ell\in \mathcal{L}_{\delta}} f_{\delta}^2(\ell).
\end{equation}
The implicit constants are allowed to depend on $c,d,\kappa,$ and $C_0$.
\end{lemma}
\begin{proof}
We adapt the ``warm-up" argument in \cite{BaV}. Note that
\begin{align*}
\sum_{\mathcal{L}_{\delta}} f_{\delta}^2(\ell)
&= \sum_{\ell\in \mathcal{L}_{\delta}} |\{(a,b)\in A\times A:\ a,b\in \ell^{c\delta}\}|\\
&= \sum_{a,b\in A}|\{\ell \in \mathcal{L}_{\delta}:\ \ a,b\in \ell^{c\delta}\}|
\end{align*}

Since any pair of points $a,b\in A$ with $a\neq b$ must have separation at least $\delta$ and at most $\sqrt{2}$, we can decompose
\begin{equation*}
A\times A=\mathcal{D}\cup \bigcup_{k=0}^{\lceil \log_2(1/\delta)\rceil}\mathcal{D}_{k},
\end{equation*}
 where
 $$
\mathcal{D}=\{(a,a):\ a\in A\},
$$
\begin{equation*}
\mathcal{D}_{k}:=\{(a,b)\in A\times A:\ 2^{-k}<|a-b| \leq 2^{-k+1}\}.
\end{equation*}
If $(a,b)\in \mathcal{D}_{k}$, we have
\begin{equation}\label{e-l21}
|\{\ell \in \mathcal{L}_{\delta}:\ \ a,b\in \ell^{2c\delta}\}|\sim 2^{k}.
\end{equation}
By \eqref{unconcentratedBallEstimateDiscrete} and the observation that $|A|\lesssim \delta^{-\alpha}$, we have that $|\mathcal{D}|\lesssim \delta^{-\alpha}$ and
\begin{equation}\label{e-l22}
|\mathcal{D}_{k}|\lesssim 2^{-k\alpha} \delta^{-2\alpha}.
\end{equation}

Hence,
\begin{align*}
\sum_{\mathcal{L}_{\delta}} f_{\delta}^2(\ell)
&\lesssim |\mathcal{D}|\delta^{-1}+ \sum_{k=0}^{\lceil \log_2(1/\delta)\rceil}\sum_{a,b\in \mathcal{D}_{k}}|\{\ell \in \mathcal{L}_{d,\delta}:\ \ a,b\in \ell^{2c\delta}\}|\\
&\lesssim  \delta^{-1-\alpha} + \sum_{k=0}^{\lceil \log_2(1/\delta)\rceil} 2^{-k\alpha }\delta^{-2\alpha }2^{k}\\
&\lessapprox \delta^{-1-\alpha},
\end{align*}
which proves the lemma since $|\mathcal{L}_{\delta}|\sim \delta^{-2}$.
\end{proof}

\begin{lemma}\label{ConeAssociatedToAPointLem}
Let $A\subset[0,1]^2$ be a discrete $(\kappa,C_0,\delta)$--unrectifiable one-set. Let $\Theta\subset [0,\pi)$ be an interval. Let
\begin{equation}\label{defnOfQaTheta}
 \mathcal{Q}_{a,\Theta} := \{ (a^\prime,\ell)\in A\times\mathcal{L}_{\delta}\colon a^\prime\neq a,\ a,a^\prime\in\ell^{c\delta},\ \theta(\ell)\in\Theta\}.
\end{equation}
Then
\begin{equation}
 |\mathcal{Q}_{a,\Theta}| \lessapprox |\Theta|^\kappa\delta^{-1}.
\end{equation}
\end{lemma}

\begin{proof}

The argument is similar to the proof of Lemma \ref{BoundOn2NormOfF}. For $k=0,1,$..., $\lceil \log_2(1/\delta)\rceil$, define
\begin{equation*}
A_k(\Theta) := \{a^\prime \in [0,1]^2\colon (a',\ell)\in \mathcal{Q}_{a,\Theta}\text{ for some }\ell\in\mathcal{L}_\delta,
\ |a-a'|\sim 2^{-k}\}.
\end{equation*}
Then $A_k(\Theta)$ is contained in $O(1)$ rectangles $R_{k,i}$ of dimensions $2^{-k}\times 2^{-k} |\Theta|$. By \eqref{defnOfOneDimUnrecEqnDiscrete},
\begin{equation*}
|A\cap R_{k,i}|\lesssim 2^{- k} |\Theta|^\kappa|A|.
\end{equation*}
If $a^\prime\in A_k(\Theta),$ then there are $O(2^k)$ lines $\ell$ such that $(a^\prime,\ell)\in\mathcal{Q}_{a,\Theta}.$ Since $|A|\sim\delta^{-1}$, we thus have
\begin{equation}
 \begin{split}
  |\mathcal{Q}_{a,\Theta}| &\lesssim \sum_{k=1}^{\lceil\log_2(1/\delta)\rceil}2^k 2^{- k} |\Theta|^\kappa|A|\\
  &\lessapprox |\Theta|^{\kappa}\delta^{-1}.\qedhere
 \end{split}
\end{equation}
\end{proof}
\begin{lemma}\label{ConeAssociatedToAPointSmallAlphaLem}
For $\alpha<1$, let $A\subset[0,1]^2$ be a discrete $(\alpha,C_0,\delta)$--set that is unconcentrated on lines. Let $\Theta\subset [0,\pi)$ be an interval. Let $\mathcal{Q}_{a,\Theta}$ be as defined in \eqref{defnOfQaTheta}.
Then
\begin{equation}
 |\mathcal{Q}_{a,\Theta}| \lessapprox |\Theta|^{1-\alpha}\delta^{-1}.
\end{equation}
\end{lemma}
\begin{proof}
The proof is similar to that of Lemma \ref{ConeAssociatedToAPointLem}. Define $A_k(\Theta)$ as in Lemma \ref{ConeAssociatedToAPointLem}, and note that by \eqref{unconcentratedBallEstimateDiscrete}, in place of \eqref{defnOfOneDimUnrecEqnDiscrete} we have
\begin{equation}
 |A\cap R_{k,i}|\lesssim 2^{- \alpha k} |A|.
\end{equation}
Now, if $a^\prime\in A_k(\Theta),$ then there are $O(\min(2^k,\delta^{-1}|\Theta|))$ lines $\ell$ such that $(a^\prime,\ell)\in\mathcal{Q}_{a,\Theta}.$ Since $|A|\sim\delta^{-\alpha}$, we thus have
\begin{equation}
 \begin{split}
  |\mathcal{Q}_{a,\Theta}| &\lesssim \sum_{k=1}^{\lceil\log_2(1/\delta)\rceil}\min(2^k,\delta^{-1}|\Theta|) 2^{- \alpha k} \delta^{-\alpha}\\
  &\lessapprox|\Theta|^{1-\alpha}\delta^{-1}.\qedhere
 \end{split}
\end{equation}

\end{proof}

\begin{definition}
Let $A\subset[0,1]^2$ be a discrete $(\alpha,C_0,\delta)$--set that is unconcentrated on lines. If $\Theta\subset S^1$ is an interval and $a\in B(0,d)\backslash A$, we define
\begin{equation*}
\Gamma(a,\Theta):=\{a'\in B(0,d)\backslash \{a\}:\ \frac{a'-a}{|a'-a|}\in\Theta\},
\end{equation*}
\begin{equation*}
 \mass(a;\Theta) := \sum_{\substack{\ell\in\mathcal{L}_{\delta}\colon a\in \ell^{2\delta},\\ \theta(\ell)\in\Theta\cup(\Theta+\pi)}}f_{\delta}(\ell).
\end{equation*}
\end{definition}
Note that if $\operatorname{dist}(a,A)\sim 1$ then heuristically, we have the equivalence $\mass(a;\Theta)\sim|A\cap (\Gamma(a,\Theta)\cup \Gamma(a,\Theta+\pi))|$. More precisely, we have the bounds
\begin{equation}\label{massEquivalence}
|A\cap (\Gamma(a,\Theta)\cup \Gamma(a,\Theta+\pi))|\lesssim \mass(a;\Theta) \lesssim |A\cap (\Gamma(a,\Theta^{2\delta})\cup \Gamma(a,\Theta^{2\delta}+\pi))|,
\end{equation}
where $\Theta^{2\delta}=\Theta+[-2\delta,2\delta]$ is the $2\delta$--neighborhood of $\Theta$, and where the implicit constants depend on $\dist(a,A)$.

\begin{lemma}\label{IntervalsAssociatedToPoint}
Let $A\subset[0,1]^2$ be a discrete $(\alpha,C_0,\delta)$--set that is unconcentrated on lines. Let $B_0$ be a ball of radius $r\sim 1$ such that $B_0\subset B(0,d)$ and $\dist(B_0,A)\geq r$.
Then there exists a $k\sim 1$ so that for $a\in B_0$, there exist two angular intervals $\Theta_{a,1},\Theta_{a,2}\subset [0,2\pi)$ such that:
\begin{enumerate}
 \item \label{intervalProperty1}\label{intervalFirstProperty} $|\Theta_{a,i}|=2\pi/k$.
 \item \label{intervalProperty2} $\operatorname{dist}(\Theta_{a,1},\Theta_{a,2})\geq 2\pi/k$ and $\operatorname{dist}(\Theta_{a,1},\Theta_{a,2}+\pi)\geq 2\pi/k$ (here, addition is performed on $S^1$, and $\dist(\cdot,\cdot)$ measures distance on $S^1$)
 \item \label{intervalProperty3}\label{intervalLastProperty} $\mass(a;\Theta_{a,1}) > |A|/10k,\ \mass(a;\Theta_{a,2}) >|A|/10k.$
\end{enumerate}
Moreover, we may choose $\Theta_{a,1}$ and $\Theta_{a,2}$ to be intervals whose endpoints are fractions with denominator $k$. More precisely, we may write $\Theta_{a,1}=\Theta_{i_1(a)}$ and $\Theta_{a,2}=\Theta_{i_2(a)}$, where
$\Theta_i=[\frac{2\pi(i-1)}{k}, \frac{2\pi i}{k})$ for $i=1,\dots,k$.
\end{lemma}
\begin{proof}
Let $k$ be an even integer greater than 10 and large enough so that
\begin{equation}\label{e-sectors}
|A\cap S_i|\leq \frac{1}{10}|A|, \ i=1,\dots,k,
\end{equation}
where $S_i=\Gamma(a,\Theta_i)$. This is possible by \eqref{unconcentratedLineEstimateDiscrete}.
Let
\begin{equation*}
J=\{i\in\{1,\dots,k\}:\ |A\cap S_i|\geq \frac{1}{10k}|A|\}.
\end{equation*}
Then
\begin{equation*}
\Big|\bigcup_{i\in J} A\cap S_i\Big|\geq |A|-
\Big|\bigcup_{i\notin J} A\cap S_i\Big|
\geq |A|-\frac{1}{10}|A|\geq \frac{9}{10}|A|,
\end{equation*}
so that by (\ref{e-sectors}), $|J|\geq 9$. Choose $i_1\in J.$ There are exactly 5 intervals $\Theta_i$ with $i\neq i_1$ such that $\dist(\Theta_{i_1},\Theta_i)<2\pi/k$ or $\dist(\Theta_{i_1}+\pi,\Theta_i)<2\pi/k$. Since $|J|\geq 9$, we can select $i_2\in J$ so that the intervals $\Theta_{a,1}=\Theta_{i_1},\Theta_{a,2}=\Theta_{i_2}$ satisfy conclusions 1 and 2 of the lemma. Conclusion 3 follows from the definition of $J$ and \eqref{massEquivalence}.
\end{proof}

We end with the following result.

\begin{proposition}\label{diffeo}
Let $U,V\subset\rr^2$, and let $\phi:U\to V$ be a $C^2$ diffeomorphism. Let $F\subset U$ be a compact set, and
let $A\subset F$ be a discrete $(\kappa,C,\delta)$--unrectifiable one-set. Then $\phi(A)$ is a discrete $(\kappa/2, C', \delta')$--unrectifiable one-set with $C'\sim C$ and $\delta'\sim \delta$. (All implicit constants may depend on $C$, $\phi$ and $F$, but not on $\delta$.)
\end{proposition}

\begin{proof}
Let
$A'=\phi(A)$. Since $\phi$ is a diffeomorphism, $|A^\prime|=|A|\sim\delta^{-1}$. Furthermore, since $\phi$ and $\phi^{-1}$ have bounded Jacobians on $F$ and its image respectively, the set $A'$ is $\delta'$-separated for some $\delta'\sim \delta$.

The main issue is to check
(\ref{defnOfOneDimUnrecEqnDiscrete}). Let $R'\subset V$ be a rectangle with side-lengths $0<r_1\leq r_2\leq 1$,
and let $\gamma'$ be the long axis of $R'$ (so that $\gamma'$ is a line segment of length $r_2$). Then $\gamma=\phi^{-1}(\gamma')$ is a $C^2$ curve of length $\sim r_2$. Furthermore, at each point $x\in\gamma$, the curvature of $\gamma$ at $x$ is bounded by some constant $c$. In particular, there exists a constant $c_1>0$ so that the following holds. For each point $z\in\gamma$, apply a translation and rotation so that $z$ is the origin and $\gamma$ is tangent to the $e_1$ direction at $z$. Then in a neighborhood $B(0,c_1)$ of the origin we may write $\gamma$ as the graph of the function $\gamma(t)$, where $|\gamma(t)|\leq ct^2$.
The constants $c$ and $c_1$ depend only on the first and second order derivatives of $\phi$; since $\phi$ is a $C^2$ diffeomorphism, we may choose $c$ and $c_1>0$ independent of $z$ and $\gamma$.
We also have that $\phi^{-1}(R')$ is contained in a $c_2r_1$-neighbourhood of $\gamma$, with, again, $c_2$ independent of the choice of $R'$.

Assume first that the rectangle has large eccentricity, in the sense that $r_2 \geq r_1^{1/2}$. Assume furthermore that $r_1<1$ is small enough relative to $c,c_1,c_2$, since otherwise (\ref{defnOfOneDimUnrecEqnDiscrete}) follows trivially if $C'$ is large enough. We may then cover $\phi^{-1}(R')$ by $O(r_2r_1^{-1/2})$ rectangles $R_i$ of dimensions $10 c_2 r_1\times c_3 r_1^{1/2}$ whose long axes are tangent to $\gamma$.
(The constant $c_3$ depends only on $c,c_1,c_2$, e.g.  we may take $c_3=\min(c_1,(c_2/c)^{1/2})$.)
By \eqref{defnOfOneDimUnrecEqn} applied to each $R_i$, we have
\begin{align*}
|A^{\prime} \cap R' | &\lesssim |A \cap \bigcup_i R_i |\\
&\lesssim \sum_i |A \cap R_i |\\
&\lesssim   r_2r_1^{-1/2}   \cdot r_1^{\kappa}r_1^{(1-\kappa)/2} |A|\\
&=   r_1^{\kappa/2} r_2 |A|\\
&\lesssim   r_1^{\kappa/2} r_2^{1-\kappa/2} |A^{\prime}|.
\end{align*}
If $r_2\leq r_1^{1/2}$, we instead cover $\phi^{-1}(R')$ by $O(1)$ rectangles of dimensions $r_1\times  r_2$, and then a similar calculation shows that
\begin{equation*}
|A^{\prime} \cap R' |\lesssim  r_1^{\kappa}r_2^{1-\kappa} |A^{\prime}|,
\end{equation*}
which is better than required.

\end{proof}
Using Lemma \ref{discreteCtsEquivalenceLem}, we have the following corollary of Proposition \ref{diffeo}:
\begin{cor}\label{diffeoCor}
Let $U,V\subset\rr^2$, and let $\phi:U\to V$ be a $C^2$ diffeomorphism. Let $F\subset U$ be a compact set, and
let $\mathcal{A}\subset F$ be a $(\kappa,C,\delta)$--unrectifiable one-set. Then $\phi(\mathcal A)$ is equivalent to a $(\kappa/2,\tilde C,\delta)$--unrectifiable one-set with $\tilde C\sim C$.
(All implicit constants may depend on $C$, $\phi$ and $F$, but not on $\delta$.)
\end{cor}

\subsection{Proof of Theorem  \ref{boundOnHighMultPtsBallThm}}\label{L2MethodsSection}
Let $C_2$ be a sufficiently large constant (to be chosen later), and let $G$ be a maximal $\delta$--separated subset of the set
\begin{equation*}
 \{a\in B_0 \colon \vis_{\delta}(a;A) < C_2 \delta^{-1} \lambda\}.
\end{equation*}
In order to prove Theorem \ref{boundOnHighMultPtsBallThm}A, it suffices to establish
\begin{equation}\label{discretizedLambdaIneq}
|G|\lessapprox \delta^{-4+2\alpha}\lambda^{2}, 
\end{equation}
while to prove Theorem \ref{boundOnHighMultPtsBallThm}B, we must establish
\begin{equation}\label{discretizedLambdaIneqTwo}
|G|\lessapprox \delta^{-4+2\alpha}\lambda^{2+\epsilon_1}.
\end{equation}

Note that $\vis_{\delta}(a;A)\geq 1$ for all $a\in U$, and thus \eqref{discretizedLambdaIneq} and \eqref{discretizedLambdaIneqTwo} are trivial for $\lambda < \delta/C_2$. Thus in everything that follows we shall assume $\lambda\geq \delta/C_2$. We will also require that $\delta$ be ``sufficiently small,"  meaning that $\delta<\delta_0$ for some $\delta_0$ depending only on $d,c,\kappa,\alpha$ and $C_0$. For $\delta>\delta_0$, Theorem \ref{boundOnHighMultPtsBallThm} holds trivially, provided we select sufficiently large implicit constants in \eqref{boundOnHighMultPtsInBall} and \eqref{boundOnHighMultPtsCtsEps-b}.

\subsubsection{Reduction to a bilinear setting}\label{sec-bilinear}
Recall Lemma \ref{IntervalsAssociatedToPoint}, which associates two intervals $\Theta_{a,1}$ and $\Theta_{a,2}$ to each point $a\in G$. Both $\Theta_{a,1}$ and $\Theta_{a,2}$ are of the form $[\frac{2\pi(i-1)}{k}, \frac{2\pi i}{k}],$ where $1\leq i\leq k$ and $k\sim 1$. Thus, after pigeonholing, we can find a refinement $\tilde G \subset G$ with $|\tilde G| \gtrsim |G|$ so that every point $a\in\tilde G$ has the same $\Theta_{a,1}$ and $\Theta_{a,2}$. Denote these intervals $\Theta_1$ and $\Theta_2$.

Let $a\in \tilde G$. We have
\begin{equation}
\mass(a;\Theta_1)\sim\delta^{-\alpha},\label{boundOnMassOfLinesHittingPt}
\end{equation}
\begin{equation} \label{boundOnNumberLinesHittingPt}
 \begin{split}
|\mathcal{L}_\delta(a;\Theta_1)| \leq \vis_{\delta}(a;A)
\lesssim\delta^{-\alpha}\lambda,
 \end{split}
\end{equation}
where
\begin{equation*}
\mathcal{L}_\delta(a;\Theta_1):=\{\ell\in\mathcal{L}_{\delta}\colon a\in \ell^{2\delta},\ \theta(\ell)\in\Theta_1
\cup(\Theta_1+\pi),\  \ell^{c\delta}\cap A\neq\emptyset\}.
\end{equation*}

Call a point $x\in A$ \important{good} if $x\in \ell^{c\delta}$ for some $\ell\in \mathcal{L}_\delta(a;\Theta_1)$
with $ f_\delta(\ell)>c_1\delta^{1-\alpha}\lambda^{-1}$. Then the number of points of $A\cap\Gamma(a;\Theta_1)$ that are not good is
$O(c_1\delta^{1-\alpha}\lambda^{-1}\delta^{-1}\lambda)= O(c_1\delta^{-\alpha}),$ where the implicit constant in the $O(\cdot)$ notation depends on the implicit constants in \eqref{boundOnMassOfLinesHittingPt} and \eqref{boundOnNumberLinesHittingPt}. We may therefore choose $c_1\sim 1$ small enough so that
\begin{align*}
\delta^{-\alpha}&\lesssim |\{x\in A\cap \Gamma(a;\Theta_1):\ x\text{ is good}\}|
\\
&=\sum_{j:\ c_1\delta^{1-\alpha}\lambda^{-1}\leq 2^j\lesssim \delta^{-1}} |\{x\in A: \ x\in\ell^{c\delta}\text{ for some } \ell\in \mathcal{L}_\delta(a,\Theta_1),
\ f_\delta(\ell)\sim 2^j\}|
\\
&\lesssim \sum_{j:\ c_1\delta^{1-\alpha}\lambda^{-1}\leq 2^j\lesssim \delta^{-1}} 2^j |\{\ell\in \mathcal{L}_\delta(a,\Theta_1),
\ f_\delta(\ell)\sim 2^j\}|.
\end{align*}

 Thus there exists a number $\mu^{(a)}_1\gtrsim \delta^{1-\alpha}\lambda^{-1}$ such that
\begin{equation}\label{defnOfMu}
|\{ \ell\in\mathcal{L}_{\delta}(a,\Theta_1):\  \mu^{(a)}_1<f_{\delta}(\ell) \leq2 \mu^{(a)}_1\}| \gtrapprox \delta^{-\alpha} (\mu^{(a)}_1)^{-1}.
\end{equation}

A similar argument holds for the directions contained in $\Theta_2$. We will call the corresponding quantity $\mu^{(a)}_2$. After a further pigeonholing (entailing a further refinement of $\tilde G$ by a factor of $|\log\delta|^2$), we can assume that every point $a\in \tilde G$ has common values of $\mu_1$ and $\mu_2$.

\subsubsection{Proof of Theorem \ref{boundOnHighMultPtsBallThm}A}
Note that
\begin{equation}\label{mu1Mu2GeqLambda}
 \mu_1,\mu_2\gtrsim\delta^{1-\alpha}\lambda^{-1}.
\end{equation}

Let
\begin{equation*}
\mathcal{H}_1 := \{ \ell\in\mathcal{L}_{\delta}\colon \theta(\ell)\in \Theta_1\cup (\Theta_1+\pi),\ \mu_1< f_{\delta}(\ell) \leq 2 \mu_1 \},
\end{equation*}
and define $\mathcal{H}_2$ similarly, with $\Theta_2$ in place of $\Theta_1$. From \eqref{L2NormOfFndelta} we have
\begin{equation*}
\begin{split}
\norm{f_{\delta}}_2^2 & \geq \frac{1}{|\mathcal{L}_{\delta}|} \sum_{\ell\in\mathcal{H}^{(1)}} f_{\delta}^2(\ell)\\
&\gtrsim \delta^2\mu_1^2|\mathcal{H}_1|.
\end{split}
\end{equation*}
Thus by Lemma \ref{BoundOn2NormOfF}
\begin{equation}\label{upperBoundOnH1}
\begin{split}
|\mathcal{H}_1| & \lesssim \delta^{-2}\mu_1^{-2}\norm{f_{\delta}}_2^2\\
&\lessapprox \delta^{-1-\alpha}\mu_1^{-2},
\end{split}
\end{equation}
and similarly for $\mathcal{H}^{(2)}$. Note that if $\ell_1\in \mathcal{H}_1$ and $\ell_2\in \mathcal{H}_2$, then $\angle(\ell_1,\ell_2)\geq \frac{2\pi}{k}$ (where $k\sim 1$ is the quantity from Lemma \ref{IntervalsAssociatedToPoint}), so $|\ell_1^{3\delta}\cap \ell_2^{3\delta}|\lesssim \delta^2$. Thus
\begin{equation}\label{aBilinearBound}
\begin{split}
\Big\Vert \Big(&\sum_{\ell_1\in\mathcal{H}_1} \chi_{\ell_1^{3\delta}} \Big)\Big(\sum_{\ell_2\in\mathcal{H}_2} \chi_{\ell_2^{3\delta}} \Big)\Big\Vert_{1} = \sum_{\substack{\ell_1\in\mathcal{H}_1\\ \ell_2\in\mathcal{H}_2}}|\ell_1^{3\delta} \cap\ell_2^{3\delta}|\\
 & \lessapprox \delta^2\big(\delta^{-1-\alpha}\mu_1^{-2}\big) \big(\delta^{-1-\alpha}\mu_2^{-2}\big)\\
 & = \delta^{-2\alpha}\mu_1^{-2}\mu_2^{-2}.
\end{split}
 \end{equation}
On the other hand, if $a\in\tilde G,\ \ell_1\in\mathcal{H}_1,\ \ell_2\in\mathcal{H}_2,$ and $a\in\ell_1^{2\delta},\ a\in\ell_2^{2\delta}$, then $|a^\delta\cap\ell_1^{3\delta}\cap\ell_2^{3\delta}|\geq\delta^2$. Since the elements of $\tilde G$ are $\delta$--separated, we thus have
\begin{equation}\label{lowerBoundForL2Norm}
 \begin{split}
  \Big\Vert& \Big(\sum_{\ell_1\in\mathcal{H}_1} \chi_{\ell_1^{3\delta}} \Big)\Big(\sum_{\ell_2\in\mathcal{H}_2} \chi_{\ell_2^{3\delta}} \Big)\Big\Vert_{1} \\
  &\geq \delta^2 \sum_{a\in\tilde G}| \{\ell_1 \in \mathcal{H}_1\colon a\in\ell_1^{2\delta}\}|\ | \{\ell_2 \in \mathcal{H}_2\colon a\in\ell_2^{2\delta}\}|\\
  &\gtrapprox \delta^2 |\tilde G|\ (\delta^{-\alpha}\mu_1^{-1})\ (\delta^{-\alpha}\mu_2^{-1})\\
  & = \delta^{2-2\alpha}|\tilde G|\mu_1^{-1}\mu_2^{-1}.
 \end{split}
\end{equation}
On the second-to-last line, we used \eqref{defnOfMu}. Thus by \eqref{mu1Mu2GeqLambda},
\begin{equation}\label{finalBoundOnTildeG}
 \begin{split}
|\tilde G| & \lesssim \Big(\delta^{-2\alpha}\mu_1^{-2}\mu_2^{-2}\Big)/\Big(\delta^{2-2\alpha}\mu_1^{-1}\mu_2^{-1} \Big)\\
 & \lessapprox \delta^{-2}\mu_1^{-1}\mu_2^{-1} \\
  & \lessapprox \delta^{-2}(\delta^{-1+\alpha}\lambda)(\delta^{-1+\alpha}\lambda) \\
 & \lessapprox \delta^{-4+2\alpha}\lambda^2.
 \end{split}
\end{equation}
%
%

%
%
\subsubsection{Proof of Theorem \ref{boundOnHighMultPtsBallThm}B}
Theorem \ref{boundOnHighMultPtsBallThm}B, is essentially identical, except we will obtain a stronger version of \eqref{aBilinearBound}. In order to do so, we will establish the following lemma:

\begin{lemma}[Not too many low visibility points on a line]\label{notTooManyPtsOnLineThm}
There exist constants $\epsilon_0,\epsilon_1>0$ so that the following holds.
Let $A$ and $B_0$ be as in Theorem  \ref{boundOnHighMultPtsBallThm}B, and let $\ell_0$ be a line. Then for all $\lambda < \delta^{\alpha/2-\epsilon_0}$, we have
\begin{equation}\label{boundOnHighMultPtsCtsEpsOneD}
|\{a\in B_0\cap\ell_0 \colon \vis_{\delta}(a;A) < \delta^{-1}\lambda \}|\lessapprox \delta^{\alpha-1}\lambda^{1+\epsilon_1}.
\end{equation}

The implicit constants in the above inequalities depend on $c,d,C_0,$ and $\kappa,$ but not on $\delta$ or $\lambda$.
\end{lemma}

Lemma \ref{notTooManyPtsOnLineThm} will be proved in Section \ref{BourgainSumProductMethodsSection}. Now, if $\ell_1\in\mathcal{H}_1$, then each point of $\ell_1^{3\delta} \cap \tilde G$ is hit by $O(\delta^{-1}\lambda)$ lines from $\mathcal{H}_2$ (this is essentially the definition of having low visibility). Thus for each $\ell_1\in\mathcal{H}_1$, we have
\begin{equation}
\begin{split}
\sum_{\ell_2\in\mathcal{H}_2} |\ell_1^{3\delta}\cap\tilde G^{3\delta}\cap\ell_2^{3\delta}| & \lessapprox\delta^2 (\delta^{-1}\lambda)(\delta^{\alpha-2}\lambda^{1+\epsilon_1})\\
& \lesssim \delta^{\alpha-1}\lambda^{2+\epsilon_1}.
\end{split}
\end{equation}
Thus
\begin{equation}\label{aBilinearBoundImproved}
\begin{split}
 \Big\Vert \Big(\sum_{\ell_1\in\mathcal{H}_1} \chi_{\ell_1^{3\delta}\cap \tilde G^{3\delta}} \Big)\Big(\sum_{\ell_2\in\mathcal{H}_2} \chi_{\ell_2^{3\delta}\cap \tilde G^{3\delta}} \Big)\Big\Vert_{1} & \lessapprox |\mathcal{H}_1|\delta^{\alpha-1}\lambda^{2+\epsilon_1}.\\
\end{split}
\end{equation}
A similar statement holds with $\mathcal{H}_1$ and $\mathcal{H}_2$ reversed. Thus we have
\begin{equation}
\begin{split}
  \Big\Vert \Big(\sum_{\ell_1\in\mathcal{H}_1} \chi_{\ell_1^{3\delta}\cap \tilde G^{3\delta}} \Big)\Big(\sum_{\ell_2\in\mathcal{H}_2} \chi_{\ell_2^{3\delta}\cap \tilde G^{3\delta}} \Big)\Big\Vert_{1} & \lessapprox \big(|\mathcal{H}_1|\ |\mathcal{H}_2|\big)^{1/2}\delta^{\alpha-1}\lambda^{2+\epsilon_1}.
\end{split}
\end{equation}
However, the reasoning used to obtain \eqref{lowerBoundForL2Norm} actually shows
\begin{equation}
 \Big\Vert \Big(\sum_{\ell_1\in\mathcal{H}_1} \chi_{\ell_1^{3\delta}\cap \tilde G^{3\delta}} \Big)\Big(\sum_{\ell_2\in\mathcal{H}_2} \chi_{\ell_2^{3\delta}\cap \tilde G^{3\delta}} \Big)\Big\Vert_{1} \gtrapprox  \delta^{2-2\alpha}|\tilde G|\mu_1^{-1}\mu_2^{-1},
\end{equation}
so
\begin{equation}\label{secondBoundOnTildeG}
\begin{split}
 |\tilde G| & \lessapprox\frac{ \big(|\mathcal{H}_1|\ |\mathcal{H}_2|\big)^{1/2}\delta^{\alpha-1}\lambda^{2+\epsilon_1}}{\delta^{2-2\alpha}\mu_1^{-1}\mu_2^{-1}}\\
 &\lessapprox \delta^{-4+2\alpha}\lambda^{2+\epsilon_1}.
\end{split}
\end{equation}
\subsection{Bourgain Sum-Product Methods}\label{BourgainSumProductMethodsSection}
\subsubsection{Reduction to a well-separated case}
In order to obtain Lemma \ref{notTooManyPtsOnLineThm}, it suffices to establish the following lemma.
\begin{lemma}\label{notTooManyPtsOnLineBallThm}
There exist constants $\epsilon_0,\epsilon_1>0$ so that the following holds.
Let $A$ and $B_0$ be as in Theorem  \ref{boundOnHighMultPtsBallThm}B, and let $\ell_0$ be a line such that $\dist(\ell_0,\conv(A))>c$, where $\conv(A)$ is the convex hull of $A$. Then for all $\lambda < \delta^{\alpha/2-\epsilon_0}$, \eqref{boundOnHighMultPtsCtsEpsOneD} holds.
\end{lemma}
To deduce Lemma \ref{notTooManyPtsOnLineThm} from Lemma \ref{notTooManyPtsOnLineBallThm}, let $\ell_0$ be the line from Lemma \ref{notTooManyPtsOnLineThm}. Then for $c>0$ sufficiently small, $|A\cap \ell_0^c|<|A|/2$. Let $\RR^2\backslash\ell_0^c = S_1\cup S_2$, with $S_1,S_2$ connected (and convex). Without loss of generality, we can assume $|A\cap S_1| \geq |A|/4$. Note that $A\cap S_1$ is either a $(\alpha,4C_0,\delta)$--set that is unconcentrated on lines (if $\alpha<1$) or a $(\kappa,4C_0,\delta)$--unrectifiable one-set (if $\alpha=1$), and furthermore $\vis_{\delta}(a; A\cap S_1)\leq \vis_{\delta}(a; A).$ Thus we can apply Lemma \ref{notTooManyPtsOnLineBallThm} to the set $A\cap S_1$ to obtain Lemma \ref{notTooManyPtsOnLineThm}.

We will now prove Lemma \ref{notTooManyPtsOnLineBallThm}.
Let $ A,\ B_0,$ and $\ell_0$ be as in the statement of Lemma \ref{notTooManyPtsOnLineBallThm}. Let
\begin{equation}
X := \{ a\in \ell_0\cap B_0\colon \vis_{\delta}(a; A)<\delta^{-1}\lambda\}.
\end{equation}
Our goal is to show that $|X|$ is small. Suppose that for some $\epsilon_1>0$, we have
\begin{equation}\label{smallLineVisibilityAssumption}
|X|>\delta^{\alpha-1}\lambda^{1+\epsilon_1}.
\end{equation}
We will show that this contradicts Bourgain's discretized sum-product theorem if $\epsilon_1$ is too small.

\subsubsection{$X$ is well-distributed}
We will first show that if $|X|$ is sufficiently large, then it must also be well-distributed in an appropriate sense.

\begin{definition} (cf. \cite[Theorem 2]{Bourgain1})
Let $\nu$ be a probability measure on $\RR$ or $S^1$ and let $\kappa,\tau>0$. We say that $\nu$ is $(\delta,\kappa,\tau)_1$--\emph{well distributed} if we have the estimate
\begin{equation}
\nu(I) \leq |I|^{\kappa}
\end{equation}
whenever $I$ is an interval with $\delta< |I| < \delta^{\tau}$.
\end{definition}

In the following discussion, we will think of $\tau$ as being fixed (but small), and $\delta$ going to 0. Thus the implicit constants in our theorems will be allowed to depend of $\tau$, but not on $\delta$.
In order to unify the discussion of the $\alpha=1$ and $\alpha<1$ cases, we will define
\begin{equation}
 \tilde\kappa = \left\{\begin{array}{ll}
 \kappa,& \alpha=1,\\
1-\alpha,&\alpha<1.
\end{array} \right.
\end{equation}

\begin{lemma}
Let $\mathcal A,\ B_0,\ \ell_0$ and $X$ be as above. Suppose that for some $\epsilon_1>0$, \eqref{smallLineVisibilityAssumption} holds. Then $X$ supports a $(\delta,\frac{\tilde\kappa}{2},\frac{3\epsilon_1}{\tilde\kappa})_1$--well distributed probability measure $\nu_1$.
\end{lemma}
\begin{proof}
We will use repeatedly the following geometric fact: if $A,\ B_0,\ \ell_0$ are as in Lemma \ref{notTooManyPtsOnLineBallThm}, and $\ell$ is a line such that $\ell^\delta$ intersects both $\ell_0\cap B(0,d)$ and $A$, then $\ell$ makes an angle $\sim 1$ with $\ell_0$.

After dyadic pigeonholing as in the proof of \eqref{defnOfMu}, we can assume there exists a number $\mu$ with
\begin{equation}\label{sizeOfMu}
\mu \geq \delta^{1-\alpha}\lambda^{-1},
\end{equation}
so that if we define
\begin{equation}
 \mathcal H := \{ \ell\in\mathcal{L}_{\delta}\colon f_{\delta}(\ell) \sim \mu \},
\end{equation}
then there exists a refinement $X^\prime\subset X$ with $|X^\prime|\gtrapprox|X|$ so that for all $x\in X^\prime,$
\begin{equation}\label{linesHittingPta}
|\{\ell\in\mathcal H \colon x\in \ell^{2\delta}\}| \approx \delta^{-\alpha}\mu^{-1}.
\end{equation}
We will prove that if $\nu_1$ is the probability measure on $X$ given by
\begin{equation}\label{defofnu1}
\nu_1(S) = |S\cap X'|/|X'|,
\end{equation}
then $\nu_1$ is $(\delta,\frac{\tilde\kappa}{2}, \frac{2\epsilon_1}{\tilde\kappa})_1$--well distributed.

Note first that by \eqref{smallLineVisibilityAssumption}, \eqref{sizeOfMu}, and the assumption that $\lambda\gtrsim\delta$,
we have
\begin{equation}\label{lowerBoundOnX}
|X'|\gtrapprox |X|>\delta^{\alpha-1}\lambda^{1+\epsilon_1}\geq \lambda^{\epsilon_1}\mu^{-1}\gtrsim \delta^{\epsilon_1}\mu^{-1}.
\end{equation}

Let $I\subset\ell_0$ be an interval with $|I|\geq\delta$.
If $\ell\in\mathcal{H}$ and $\ell\cap \ell_0\subset I$, then $|\ell^{2\delta}\cap I|\sim\delta$. We can therefore cover $I\cap X'$ by finitely overlapping intervals of length about $\delta$ so that the number of such intervals is about $\delta^{-1}|I\cap X'|$. By \eqref{linesHittingPta}, each such interval is intersected by at most $\approx \delta^{-\alpha}\mu^{-1}$ lines of $\mathcal{H}$. Hence
\begin{equation}\label{lowerBoundOnPtsHittingInterval}
|\{\ell\in\mathcal H \colon \ell\cap \ell_0\subset I\}| \gtrapprox \delta^{-1-\alpha}\mu^{-1}|X' \cap I|.
\end{equation}
To bound the quantity on the left, we will need a lemma.

\begin{lemma}
Let $I\subset \ell_0\cap B_0$ be an interval with $|I|\geq\delta$. Then
\begin{equation}\label{boundOnNumberOfLinesHittingInterval}
  |\{\ell\in\mathcal H \colon \ell\cap \ell_0\subset I\}| \lessapprox \mu^{-2}\delta^{-1-\alpha} |I|^{\tilde\kappa}.
\end{equation}
\end{lemma}
\begin{proof}
Let
\begin{equation}
 \mathcal Q := \{(a,a^\prime,\ell)\in A\times A\times\mathcal H \colon a,a^\prime \in \ell^{c\delta},\ \ell\cap\ell_0\subset I\}.
\end{equation}
For $a\in A$, let $\Theta_a$ be the $\delta$ neighborhood of the arc $\{ \frac{a-a^\prime}{|a-a^\prime|}   :\ a^\prime\in I \}
\subset S^1$. Since $|I|\geq\delta$ and $\dist(a,I)\lesssim 1$, we have $|\Theta_a|\lesssim |I|$. Then,
\begin{equation}
 \begin{split}
  |\mathcal Q| & \leq \sum_{a\in A}|(a^\prime,\ell)\in A\times\mathcal H\colon a\neq a^\prime,\ a,a^\prime\in\ell^{c\delta},\ \theta(\ell)\in\Theta_a\} |\\
  & \lessapprox \sum_{a\in A}|\Theta_a|^{\tilde\kappa} \delta^{-1}\\
  & \lesssim |I|^{\tilde\kappa}\delta^{-1-\alpha},
 \end{split}
\end{equation}
where on the second line we used either Lemma \ref{ConeAssociatedToAPointLem} or Lemma \ref{ConeAssociatedToAPointSmallAlphaLem}, depending on whether $\alpha=1$ or $\alpha<1$. Since each $\ell \in \mathcal{H}$ meets $\sim\mu$ points in $A$ and therefore enters $\mu^2$ triples in $\mathcal{Q}$, this proves
\eqref{boundOnNumberOfLinesHittingInterval}.
\end{proof}

Combining \eqref{lowerBoundOnPtsHittingInterval} and \eqref{boundOnNumberOfLinesHittingInterval}, we have
\begin{equation}
 |X'\cap I| \lessapprox \mu^{-1}|I|^{\tilde\kappa}.
\end{equation}
Combining this with \eqref{lowerBoundOnX}, we see that if $|I| < \delta^{3\epsilon_1/\tilde\kappa}$, then
\begin{equation}
\frac{|X'\cap I|}{|X'|}  \lessapprox \delta^{-\epsilon_1}|I|^{\tilde\kappa} \leq \delta^{\epsilon_1/\tilde\kappa} |I|^{\tilde\kappa/2}.
\end{equation}
This proves that $\nu_1$ defined in \eqref{defofnu1} is well distributed as claimed.


\subsubsection{A projective transformation}\label{sec-projective}
In this subsection we will describe the projective transformation $T:\rr^2\setminus\ell_0\to\rr^2$ which ``takes the line $\ell_0$ to the line at infinity'' in the sense that if $a\in\ell_0$, then the $T$-image of the family of lines passing through $a$ is the family of parallel lines pointing in some direction $\theta_a$. We will also demonstrate that this mapping of lines is highly regular, both as a function of $a$ and as a function of the direction of the line. Ultimately, the goal is to show that this transformation must send any counterexample to Theorem \ref{boundOnHighMultPtsBallThm}
to a counterexample to Bourgain's discretized sum-product theorem (Theorem \ref{BourgainsDiscreteSumProdThm}).

Without loss of generality, we may assume that  $\ell_0$ is the $x$-axis, $B_0\cap\ell_0$ is contained in $[-10,0]\times\{0\}$, and $A\subset[1,20]^2$. (We can reduce to this case by partitioning $A$ into finitely many pieces, applying affine transformations, and noting that such transformations preserve the property that $A$ is either a discrete $(\alpha,C_0,\delta)$--set that is unconcentrated on lines (if $\alpha<1$) or a discrete $(\kappa,C_0,\delta)$--unrectifiable one-set (if $\alpha=1$). We may increase the value of $C_0$ by a constant factor if necessary. Define
\begin{equation}\label{protran-def}
 T(x,y) :=  \Big( \frac{x+1}{y}, \frac{y+1}{y}\Big).
\end{equation}
Since $T$ maps lines to lines, we may verify that after a refinement, $T(A)$ is a discrete $(\alpha,C_0,\delta)$--set that is unconcentrated on lines (at this point, this is the only property we will need, even if $\alpha=1$).

To see what $T$ does to lines, note that
\begin{equation}\label{t-lines}
T(x+t,mt)=(\frac1{m},1)+\frac1{mt}(x+1,1)
\end{equation}
This has two consequences of interest. First, fix a point $x\in [-10,0]$. Then $\vis_\delta((x,0),A)$ counts the number of lines through $(x,0)$ with $\delta$-separated slopes whose $\delta$--neighborhoods meet $A$. The mapping $T$ takes lines $\ell$ through $(x,0)$ to lines in the direction $\theta_x:=\operatorname{arccot}(x+1)$. Moreover, if $\ell,\ell'$ are two such lines with slopes $m,m'$ respectively, then their images pass through  $(\frac1{m},1)$ and $(\frac1{m'},1)$. We have placed $x$ and $A$ so that $1/40<m<30$ for all lines $\ell$ whose $\delta$--neighborhoods meet both $A$ and $(x,0)$. Hence, the distance between $T(\ell)$ and $T(\ell')$ is proportional to the acute angle between $\ell$ and $\ell'$. It follows that
\begin{equation}\label{visibilityComparable}
 \vis_{\delta} (x;A)\sim \delta^{-1}|\pi_{\theta_x} (T(A^{\delta}))|,
\end{equation}
where $A^{\delta}$ is the $\delta$--neighborhood of $A$. The same estimate holds if we replace $A^{\delta}$ by $\mathcal{A}_1$, where $\mathcal{A}_1$ is a union of $\delta$--squares centered at the points of $A$. Note that since $A$ satisfies \eqref{unconcentratedBallEstimateDiscrete}, $\mathcal{A}_1$ satisfies $\eqref{unconcentratedBallEstimate}$ for some constant $C^\prime$ that is comparable to $C$, and $|\mathcal{A}_1|\sim \delta^2|\mathcal{A}|\sim\delta^\alpha$.

Furthermore, the map $x\to\theta_x$ has the property that if $x,x^\prime\in[-10,0]$, then
\begin{equation}\label{distanceComparableAngularDistance}
  |\theta_x-\theta_{x^\prime}|\sim|x-x^\prime|.
\end{equation}
Let $\nu$ be the push-forward measure of $\nu_1$ (so that $\nu$ is supported on $S^1$).
Then by \eqref{distanceComparableAngularDistance}, $\nu$ is $(\delta,\frac{\tilde\kappa}{4},\frac{3\epsilon_1}{\tilde\kappa})_1$--well distributed.
%

In light of \eqref{visibilityComparable}, our low visibility assumption implies that
for all $\theta \in \supp\nu,$ we have
 \begin{equation}\label{smallProjections}
 \begin{split}
 |\pi_\theta(\mathcal{A}_1)| 
 & \lesssim \delta^{\alpha/2-3\epsilon_0}.
 \end{split}
 \end{equation}
 We will now proceed to obtain a contradiction.


\subsubsection{Bourgain's sum-product theorem, and a contradiction}

We shall now state a version of Bourgain's discretized sum-product theorem.

\begin{theorem}[Bourgain, \cite{Bourgain1}, Theorem 3]\label{BourgainsDiscreteSumProdThm}
Given $0<\alpha<2,$ $\beta>0$, and $\kappa>0$, there exists $\tau_0>0$ and $\eta>\alpha/2$ such that the following holds for all $\delta>0$ sufficiently small.

Let $\mu_1$ be a $(\delta,\kappa, \tau_0)_1$--well distributed probability measure on $S^1$. Let $\mathcal A \subset[1,2]^2$ be a union of $\delta$--squares with the property that
\begin{equation}
|\mathcal{A}|\sim \delta^{2-\alpha},
\end{equation}
and
\begin{equation}\label{BourgainDistributednessProperty}
 |\mathcal A \cap B| \leq \rho^{\beta}|\mathcal{A}|
\end{equation}
whenever $B$ is a ball of radius $\rho$ with $\delta<\rho<1$.

Then there exists $\theta\in\supp\mu_1$ such that
\begin{equation}\label{bourgainLargeProjection}
|\pi_\theta(\mathcal A)|>\delta^{1-\eta}.
\end{equation}
\end{theorem}

In the statement of Theorem \ref{BourgainsDiscreteSumProdThm} in \cite{Bourgain1}, Bourgain has the more restrictive requirement that $\mu_1$ be a $(\delta,\kappa, 0)_1$--well distributed probability measure on $S^1$ (i.e.~that the well-distributedness property hold for all intervals, not just those of length at most $\delta^{\tau_0}$). However, the remark on page 221 of \cite{Bourgain1} observes that the proof of Theorem \ref{BourgainsDiscreteSumProdThm} only requires $\mu_1$ to be $(\delta,\kappa, \tau_0)_1$--well distributed. 

We will apply Theorem \ref{BourgainsDiscreteSumProdThm} with $\mathcal A_1$ in place of $\mathcal A$, $\nu$ in place of $\mu_1$, $\tilde\kappa/4$ in place of $\kappa$, $\delta^\prime$ in place of $\delta$, $\alpha$ as specified in the statement of Theorem \ref{boundOnHighMultPtsThm}, and $\beta=\alpha$. Select $\epsilon_0$ and $\epsilon_1$ sufficiently small, so that $4\epsilon_0<\eta-\alpha/2$ and $\frac{3\epsilon_1}{\kappa}\leq\tau_0$. By \eqref{bourgainLargeProjection}, there exists $\theta\in\supp\nu$ with
\begin{equation}
\begin{split}
 |\pi_{\theta}(\mathcal A_1)|& \geq\delta^{1/2 - (\eta-\alpha/2)}\\
 &\geq \delta^{1/2 - 4\epsilon_0}.
\end{split}
 \end{equation}
This contradicts \eqref{smallProjections}. Therefore, \eqref{smallLineVisibilityAssumption} cannot hold.\end{proof}


\section{Pointwise upper bound}\label{upper}
We first establish some tools and terminology for self-similar sets. In this section, $\mathcal{J}$  will be a one-dimensional self-similar set with no rotations and with equal contraction ratios, i.e.~$\mathcal{J}$ satisfies \eqref{defnOfJSelfSim} where for each $i=1,\ldots,s$ we have $\mathcal{O}_i = I$ and $\lambda_i = \frac{1}{s}$. Without loss of generality we can assume that $\diam\mathcal{J}\sim 1$. We also fix $a=0$.

Let $W_n=\{1,...,s\}^n$ be the set of all words of length $n$ in the alphabet $\{1,...,s\}$, and $W=\bigcup_{n=0}^\infty W_n$, where $W_0$ consists of the empty word. For $w=(w_1,...,w_n)\in W_n$, let
\begin{equation*}
\begin{split}
T_w&:=T_{w_n}\circ\dots\circ T_{w_1},\\
\mathcal{Q}_n&:=\{T_w(\mathcal{J}_0):w\in W_n\}.
\end{split}
\end{equation*}
We refer to the $Q\in\mathcal{Q}_n$ as $\important{disks at stage }n$, each associated with a word $w=w(Q)\in W_n$ so that $T_w(\mathcal{J}_0)=Q$.

For $w=(w_1,...,w_n)$ and $w'=(w_{n+1},...,w_{n+m})$, let $ww':=(w_1,...,w_{n+m})$. If $Q$ and $Q'$ are associated with the words $w$ and $ww'$ respectively, we say that $Q'$ is a \important{descendant of} $Q$ \important{of generation} $n+m$, and $Q$ is an \important{ancestor of} $Q'$ \important{of generation} $n$. In particular, $Q'\subseteq Q$.

We also employ the notations $Q'>_m Q$, $Q<_m Q'$.
We will sometimes refer to descendants  and ancestors with $m=1$ as children and parents respectively.
Finally, we write $Q\prec Q'$ and $w\prec w'$ if the word $w$ associated to $Q$ is a sub-word of the word $w'$ associated to $Q'$, that is, $w'=vwv'$ for some $v,v'\in W$.

Figures \ref{greenL}, \ref{greenM}, and \ref{greenN} should be kept in mind whenever $<$, $\prec$ are invoked.

\begin{figure}[htbp]\centering\includegraphics[width=5.in]{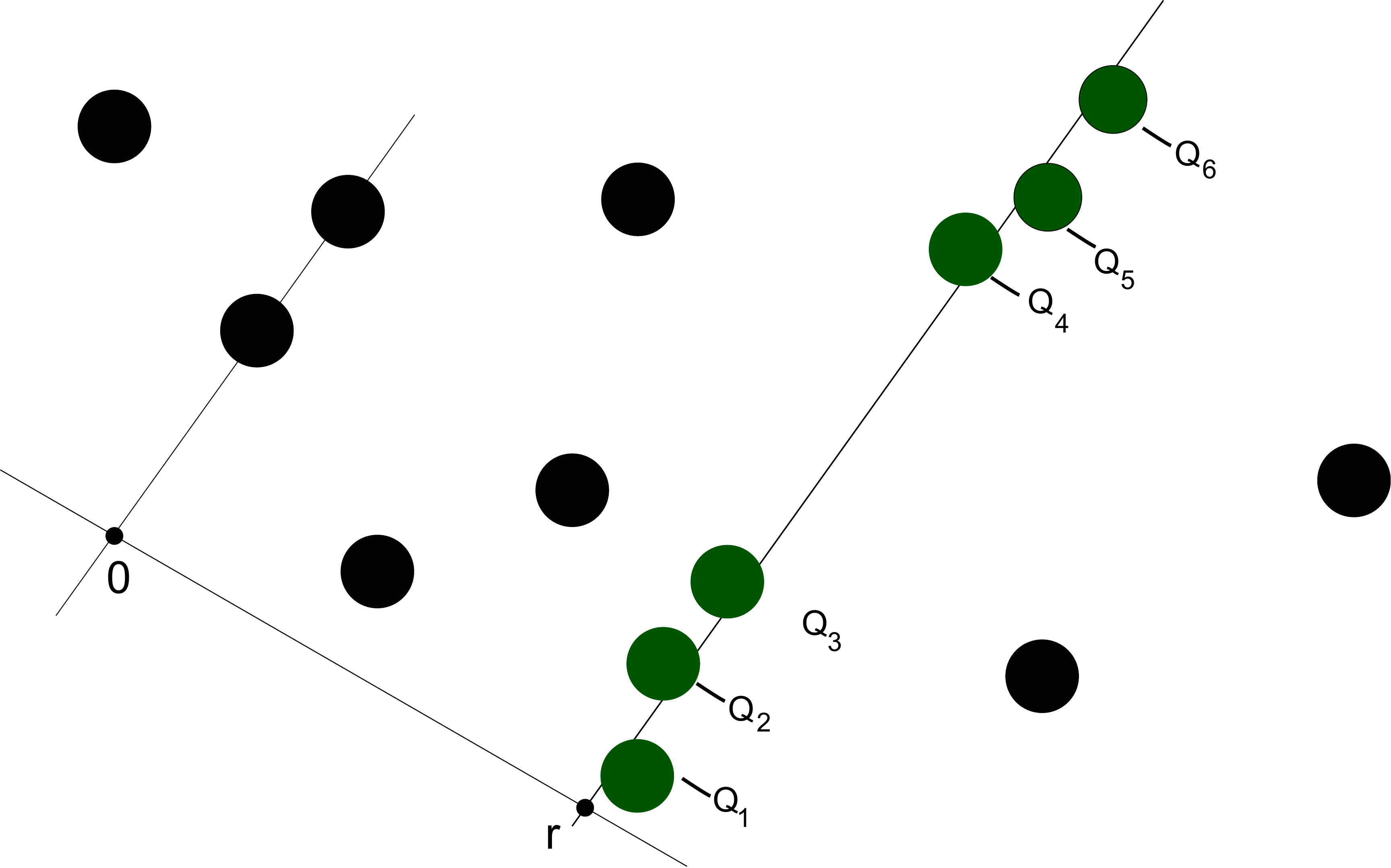}\caption{The disks $Q_1,...,Q_6\in \mathcal{Q}_L$ form a tall stack.}\label{greenL}\end{figure}

\begin{figure}[htbp]\centering\includegraphics[width=5.in]{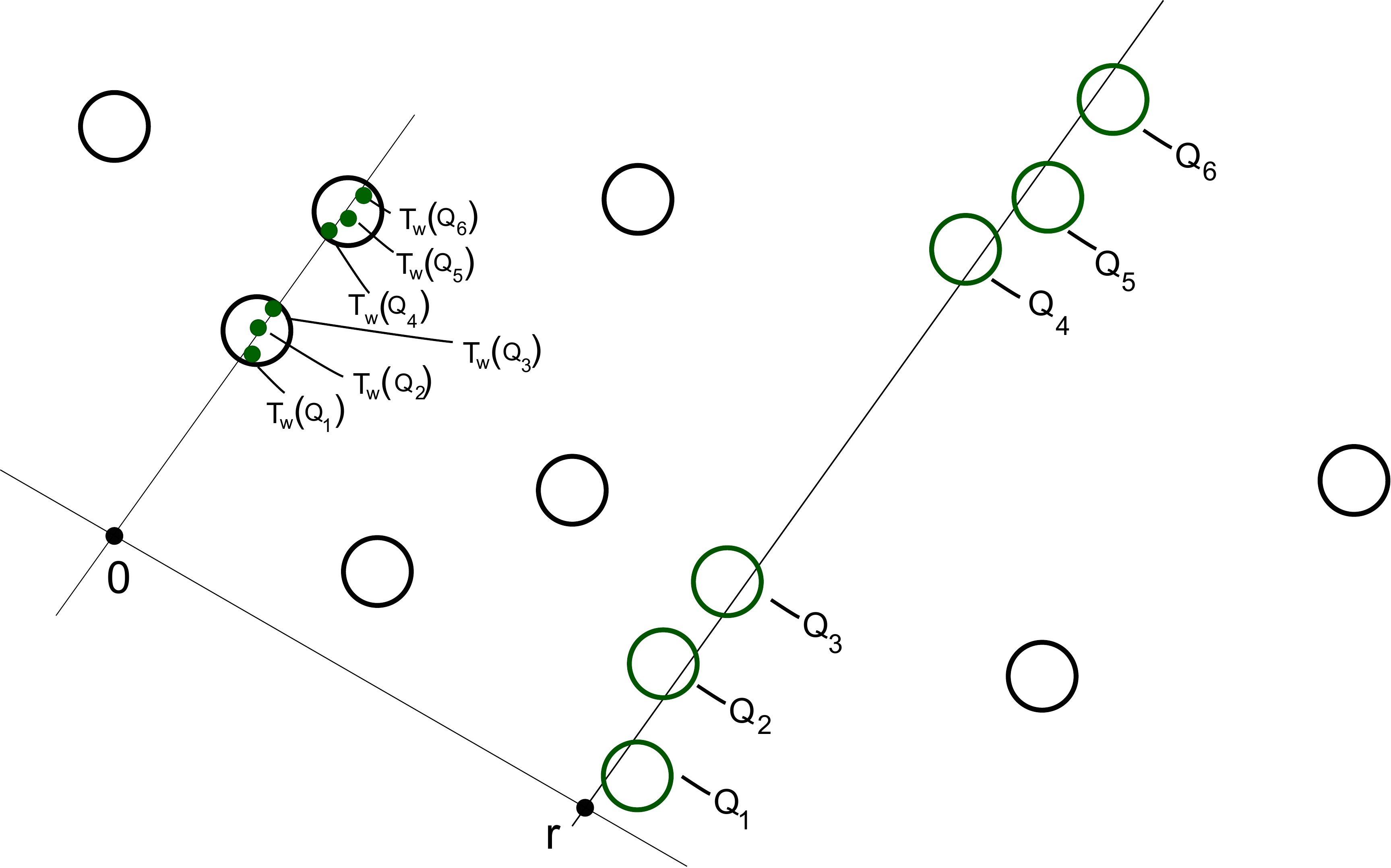}\caption{Continuing from Figure \ref{greenL}, this is a rough partial sketch of $\mathcal{J}_M$, $M>L$. $T_w(Q_1),...,T_w(Q_6)$ are singled out, for some $w\in W_{M-L}$.}\label{greenM}\end{figure}

\begin{figure}[htbp]\centering\includegraphics[width=3.in]{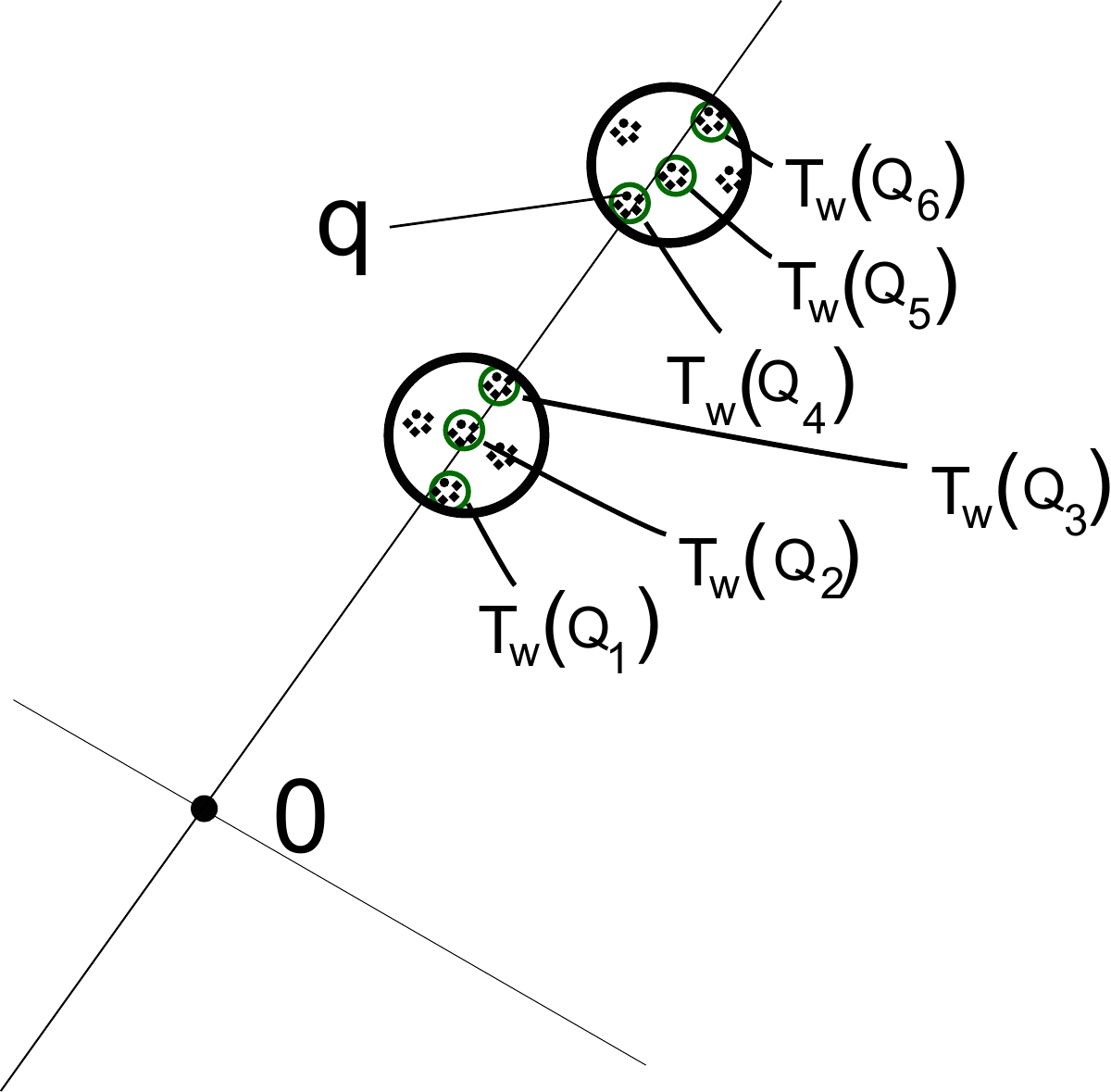}\caption{A closer look at a smaller portion of $\mathcal{J}_n$, $n>M>L$. Since $T_w(Q_4)<q$, it follows that $Q_4\prec q$. In words, ``$q$ is a descendant of a self-similar copy of $Q_4$.''}\label{greenN}\end{figure}

The \important{projection counting function}, $f_{n,\theta}:\rr\to\nn$ (analogous to $f_\delta(\ell)$ from Section \ref{defnOfFSection}), is
\begin{equation*}
f_{n,\theta}:=\sum_{Q\in\mathcal{Q}_n} \chi_{\pi_\theta(Q)}.
\end{equation*}
Large values of $f_{n,\theta}$ indicate large concentrations of $L^1$ mass on a small set, so that the support of $f_{n,\theta}$ cannot be too large; note that $\text{supp}f_{n,\theta}=\pi_\theta{\mathcal{J}_n}$. The former statement is quantified using the Hardy-Littlewood operator $M$ and self-similarity. The advantage in applying the Hardy-Littlewood operator is that, while $f_{n,\theta}(r)$ need not be increasing in $n$ as $\theta$ and $r$ are fixed, $Mf_{n,\theta}$ is much better behaved, in a way that we now quantify. Practically, we may treat it as nondecreasing in $n$.

\begin{lemma}\label{HL-1}
Let $n,N\in\nn$ with $n\leq N$. If $f_{n,\theta}(r)\geq K$, then $Mf_{N,\theta}(r')\geq\frac{K}{2}$ whenever $r,r'$ belong to the same $Q\in\mathcal{Q}_n$.
\begin{proof}The Hardy-Littlewood estimate is obtained using the interval $I=[r-|\pi_\theta(\mathcal{J}_0)| s^{-n},r+|\pi_\theta(\mathcal{J}_0)| s^{-n}]$. This interval contains at least $K$ projected disks $\pi_\theta(Q_i)$, where $Q_i\in\mathcal{Q}_n$. By induction on $m$, it can be shown that for each $i$,
\begin{equation*}
\int_I\sum_{Q'>_m Q_i} \chi_{\pi_\theta(Q')}(r)\, dr = \int_I \chi_{\pi_\theta(Q_i)}(r)\, dr\geq\frac{|I|}{2}.
\end{equation*}
Applying this to each $Q_i$ when $m=N-n$, summing over $i$, and dividing by $|I|$ establishes the claim.
\end{proof}
\end{lemma}

\begin{definition}
For fixed $\theta$, given $K>0$, we say that $Q\in\mathcal{Q}_N$ is \important{$K$-stacked} if $Mf_{N,\theta}(r)\geq K$ for \allEvery{all} $r\in\pi_\theta(Q)$. We also write $Q\in\mathcal{Q}_N^S$ 
if $Q$ is $K$-stacked at the angle $\theta$.
\end{definition}

In applications, $K=K(N)$ will grow slowly with $N$. Whenever $K$ is clear from context, we will omit it and refer to $K$-stacked disks as ``stacked".

\begin{corollary}\label{HL-1-cor}
If $f_{n,\theta}(r)\geq K$ for some $r\in\pi_\theta(Q)$ and $Q\in\mathcal{Q}_n$, then $Q'$ is $K/2$-stacked at $\theta$ for all $Q'>Q$.
\end{corollary}

 We also have the following.

\begin{lemma}\label{key1}
Fix $\theta$ and $K$. If $Q$ is $K$-stacked, then $Q'$ is $\frac{K}{2}$-stacked for all $Q'\succ Q$.
\begin{proof}
Seen by self-similarly rescaling an appropriate Hardy-Littlewood interval containing $\pi_\theta(Q)$ as in Lemma \ref{HL-1}. Details omitted.
\end{proof}
\end{lemma}

The following is a blueprint for how Hardy-Littlewood estimates bound $|\text{supp}f|$. A small number of ``bad disks'' can be measured separately, leaving a good estimate based on the structure of the ``good disks''.

\begin{lemma}\label{HL-2}
For fixed $N$, $\theta$, $K$, suppose that there are at most $\frac{s^N}{K}$ unstacked disks $Q\in\mathcal{Q}_N$. Then $|\text{supp}f_N|\lesssim\frac1{K}$.
\end{lemma}
\begin{proof}
We have
\begin{equation*}
f_{N,\theta}=\sum_{Q\notin\mathcal{Q}_N^S}\chi_{\pi_\theta(Q)}+\sum_{Q\in\mathcal{Q}_N^S}\chi_{\pi_\theta(Q)}.
\end{equation*}
Note that for all $Q\in\mathcal{Q}_N^S$, $\pi_\theta(Q)\subseteq\{r:Mf_{N,\theta}(r)\geq K\}$. The above line splits the support of $f_{N,\theta}$ into two sets. Estimating the first trivially and applying the Hardy-Littlewood inequality to the second,
\begin{equation*}
 |\text{supp}f_{N,\theta}|\leq s^{-N}|\pi_\theta(\mathcal{J}_0)|\#\{\mathcal{Q}_N\setminus\mathcal{Q}_N^S\}+\frac{1}{K}||f_{N,\theta}||_1\lesssim \frac{1}{K}.\qedhere
\end{equation*}
\end{proof}

In order to apply Hardy-Littlewood analysis to the visibility integral, we will need a visibility analogue of $f_{n,\theta}$:
\begin{equation*}
g_n(\theta):=f_{n,\theta-\pi/2}(0).
\end{equation*}
Note that $\text{supp}(g_n)$ is the union of $P_0(\mathcal{J}_n)$ and the antipodal points of $P_0(\mathcal{J}_n)$, i.e.~$\text{supp}(g_n)=P_0(\mathcal{J}_n)\cup(P_0(\mathcal{J}_n)+\pi)$. We also need the following fact.

\begin{lemma}\label{HL-vis}
If $\mathcal{J}\subseteq B(0,100)\setminus B(0,\frac1{100})$, then $Mf_{n,\theta}(0)\sim Mg_n(\theta+\pi/2)$.
\begin{proof}
The key observation is that that the Hardy-Littlewood intervals for the two functions are comparable up to minor dilations. Let $\Theta\subset S^1$ be an interval centered at $\theta+\pi/2$ so that $\frac{1}{|\Theta|}\int_{\Theta}g_n(t)\geq \frac{1}{4}Mg_n(\theta+\pi/2)$. We can assume that $|\Theta|\gtrsim n^{-s}$, since $g_n(t)$ is a sum of characteristic functions of intervals, each of which has length $\gtrsim n^{-s}$. Let $C\Theta$ be the $C$--fold dilate of $\Theta$, where $C=O(1)$ is chosen so that $|C\Theta|\geq 2 n^{-s}.$ Let $\Gamma\subset\RR^2$ be the intersection of $B(0,100)\setminus B(0,\frac1{100})$ with the set of all rays from the origin that make an angle $t\in C\Theta$ with the $x$--axis. Thus $\Gamma$ is a segment of an annulus. Furthermore, the number of disks $Q\in\mathcal{Q}_n$ contained in $\Gamma$ is $\gtrsim |\Theta|s^{n} Mg_n(\theta+\pi/2),$ since each disk can contribute at most $O(s^{-n})$ to the integral $\int_{\Theta}g_n(t)dt.$ Now, consider the infinite strip $A\subset\RR^2$
centered at the origin of dimensions $2t_0\times\infty$, whose long axis points in the direction $\theta+\pi/2$. Since $\Gamma\subset B(0,100),$ we can select $t_0\lesssim |\Theta|$ so that $A$ contains $\Gamma$. Let $I=[-t_0,t_0]\subset\RR$.

Recall that for any angle $t$ and for any disk $Q\in\mathcal{Q}_n$, we have $\int_{\RR}\chi_{\pi_{t}(Q)}\sim s^{-n}$. Thus,
\begin{equation*}
\begin{split}
\int_I f_{n,\theta}(r)dr & \gtrsim \sum_{\substack{Q\in\mathcal Q_n\\Q\subset A}}\chi_{\pi_\theta(Q)}(r)dr \\
& \gtrsim s^{-n}|\{Q\in\mathcal Q_n: Q\in A\}|\\
&\gtrsim s^{-n}\big(|\Theta|s^{n} Mg_n(\theta)\big).
\end{split}
\end{equation*}
Since $|I|\lesssim |\Theta|$, we have $M f_{n,\theta}\gtrsim Mg_n(\theta+\pi/2)$. A similar argument establishes the reverse quasi-inequality.
\end{proof}
\end{lemma}

The model for turning a Favard length estimate into a visibility estimate is as follows:

\begin{theorem}\label{fav-upper-heur}(Heuristic form of Theorem \ref{fav-upper})
Suppose $0\notin\mathcal{J}$. If $\mathcal{J}_L$ has small Favard length, then $\mathcal{J}_N$ has small visibility from the origin for all $N$ much larger than $L$.

\begin{proof}[Heuristic proof.]

Since $\Fav(\mathcal{J}_L)$ is small, it must be the case that for most $\theta$, there is at least one tall stack of disks in $\mathcal{Q}_L$ pointing in the $\theta$ direction. The remaining $\theta$ belong to a small set $E_L$.

In bounding Favard length, the main idea of \cite{NPV} and others is to establish that for most angles $\theta$, most disks are $K$-stacked, for $K$ appropriately large as a function of $N$. The additional problem in obtaining upper visibility bounds is that it is not enough for stacking to occur {\it somewhere} in $\mathcal{J}_N$, but instead we need stacking along lines passing through the vantage point $0$. \cite{SS} overcomes this difficulty by choosing $N$ sufficiently large (compared to $L$) and using self-similarity to prove that most disks of $\mathcal{J}_N$ descend from self-similar copies of the tall stack of $\mathcal{J}_L$. This situation is as in Figures \ref{greenM} and \ref{greenN}.

Specifically, if $N$ is sufficiently large compared to $L$, then almost all $Q'\in\mathcal{Q}_N$ satisfy $Q\prec Q'$ for all $Q\in\mathcal{Q}_L$. The set of disks in $\mathcal{Q}_N$ that fail to satisfy this property are negligible, and they do not affect our argument. If $Q'$ satisfies  $Q\prec Q'$ for all $Q\in\mathcal{Q}_L$, then we say $Q'$ is generic. In particular, if $\theta\notin E_L$ and $Q'$ is generic, then $Q'$ is stacked by Lemma \ref{key1}. Using Lemma \ref{HL-vis} to pass to the function $g_N$, we can apply Hardy-Littlewood analysis as in Lemma \ref{HL-2} to bound the projection of the generic disks.
\end{proof}
\end{theorem}

We now quantify the above argument.
First, we bound the size of $E_L$, the set of ``bad angles.''

\begin{lemma}\label{m-of-E}
Let
\begin{equation*}
\begin{split}
K&=K(L):=\frac1{\sqrt{\Fav(\mathcal{J}_L)}},\\
E_L&:=\{\theta:||f_{L,\theta-\pi/2}||_\infty\leq K\}.
\end{split}
\end{equation*}
Then
\begin{equation*}
|E_L|\lesssim\frac1{K}.
\end{equation*}

%
\end{lemma}
\begin{proof}
We have
\begin{equation*}
K^{-2}=\text{Fav}(\mathcal{J}_L)\geq \frac1{K} \Big|\{\theta:|\pi_\theta(\mathcal{J}_L)|\geq\frac1{K}\}\Big| \gtrsim \frac1{K} |E_L|.
\end{equation*}
The last inequality follows from
\begin{equation*}
1\sim\int f_{L,\theta}(x) dx\leq ||f_{L,\theta}||_\infty\cdot |\pi_\theta(\mathcal{J}_L)|.\qedhere
\end{equation*}

\end{proof}

We now fix some small $\epsilon>0$ (independent of $N$), and define
\begin{equation}\label{sizeOfL}
L:=(1-2\epsilon)\log_s N.
\end{equation}
We will say that $Q'\in\mathcal{Q}_N$ is \important{generic} if $Q\prec Q'$ for all $Q\in\mathcal{Q}_L$. Note that this definition depends on $N$ and $L$, but we will not display that dependence.
For $\mathcal{R}\subseteq\mathcal{Q}_N$, let
\begin{equation*}
\mathbb{P}(\mathcal{R}):=\frac{\# \mathcal{R}}{s^N},
\end{equation*}

\begin{equation*}
\mathcal{G}_N:=\{Q'\in\mathcal{Q}_N:Q'\text{ is generic}\}.
\end{equation*}

\begin{proposition}\label{bad words}
We have
\begin{equation*}
\mathbb{P}(\mathcal{Q}_N\setminus\mathcal{G}_N)\lesssim N^{1-2\epsilon}e^{-N^\epsilon}.
\end{equation*}
\end{proposition}
We will prove Proposition \ref{bad words} using the following lemma.
\begin{lemma}\label{boundOnBadWordProb}
Fix $w\in W_L$ and let $\beta_N(w)=\{b\in W_N\colon w\not\prec b\}$. Then
\begin{equation*}
 \mathbb{P}(\beta_N(w))\lesssim e^{-N^\epsilon},
\end{equation*}
where the value of $\epsilon$ is the same as in \eqref{sizeOfL}.
\end{lemma}
For proof, see below. Using Lemma \ref{boundOnBadWordProb}, we have
\begin{equation*}
\begin{split}
\mathbb{P}(W_N\setminus \mathcal{G}_N)&\leq\sum_{w\in W_L} \mathbb{P}(\beta_N(w))\\
&\lesssim s^L e^{-N^\epsilon}\\
&=N^{1-2\epsilon}e^{- N^\epsilon},
\end{split}
\end{equation*}
which establishes Proposition \ref{bad words}.
\begin{proof}[Proof of Lemma \ref{boundOnBadWordProb}]
Without loss of generality, we may assume that $L$ divides $N$, so that we can divide $b$ into segments of length $L$. For each of these segments, the probability that it is not equal to $w$ is $(s^L-1)/s^L=1-s^{-L}$.
Then $\mathbb{P}(\beta_N(w))$ is bounded from above by the probability that all such segments are different from $w$, so that
\begin{align*}
\mathbb{P}(\beta_N(w)) &\leq (1-s^{-L})^{N/L}\\
&=(1-N^{-(1-2\epsilon)})^{N/L}\\
&=\exp\left(\frac{N}{(1-2\epsilon)\log_sN}\,\log(1-N^{-(1-2\epsilon)})\right)
\end{align*}
Using that $\log (1-x)\approx -x$, and absorbing $\log_s N$ into the loss of $\epsilon$ in the exponent, we get the desired estimate.
(Note that a slightly improved bound can be proved using the ``longest run of heads'' estimate from \cite{ER}.)
\end{proof}
We now have all of the necessary tools to prove Theorem \ref{fav-upper}.

\begin{proof}[Proof of Theorem \ref{fav-upper}]

We first recall the estimate from Lemma \ref{m-of-E} on the size of the set of ``bad angles" $E_L$:
\begin{equation}\label{measureOfEL}
 |E_L|\lesssim\frac1{K}.
\end{equation}
We claim that
\begin{equation}\label{boundBadWords}
|P_0(\mathcal{J}_N\setminus\bigcup_{Q\in\mathcal{G}_N}Q)|\lesssim\frac1{K},
\end{equation}
\begin{equation}\label{HLBoundGoodWords}
|P_0(\bigcup_{Q\in\mathcal{G}_N}Q)\setminus E_L| \lesssim\frac1{K}.
\end{equation}

Combining \eqref{measureOfEL}, \eqref{boundBadWords}, and \eqref{HLBoundGoodWords}, we obtain
\begin{equation*}
\begin{split}
|P_0(\mathcal{J}_N)|&\leq |P_0(\mathcal{J}_N\setminus\bigcup_{Q\in\mathcal{G}_N}Q)|+|P_0(\bigcup_{Q\in\mathcal{G}_N}Q)\setminus E_L|+|E_L|\\
&\lesssim\frac1{K},
\end{split}
\end{equation*}
as required.

It remains to prove the above claims. We begin with \eqref{boundBadWords}.
Visibility is sub-additive, and $\vis(0;X)\lesssim\frac{\diam(X)}{\text{dist}(0,X)}$. Assuming $\mathcal{J}_0$ is separated from $0$ and using the fact that $\text{diam}(Q)\sim s^{-N}$ for $Q\in\mathcal{Q}_N$, it follows that
\begin{equation*}
|P_0(\mathcal{J}_N\setminus\bigcup_{Q\in\mathcal{G}_N}Q)|\lesssim \mathbb{P}(\mathcal{Q}_N\setminus\mathcal{G}_N),
\end{equation*}
where the implicit constant depends on $\dist(\mathcal{J}_0,0)$.
By (\ref{favard-lower}), we have
$$
\frac{1}{K}=\sqrt{\Fav(\mathcal{J}_L)}\gtrsim L^{-1/2}.
$$
Hence it suffices to prove that
$$
\mathbb{P}(\mathcal{Q}_N\setminus\mathcal{G}_N)\lesssim L^{-1/2}.
$$
But that follows from
Proposition \ref{bad words} and \eqref{sizeOfL}.

Finally, we prove \eqref{HLBoundGoodWords}.
Consider $\theta\in P_0(\mathcal{G}_N)\setminus E_L$. Since $\theta\notin E_L$, it follows from Corollary \ref{HL-1-cor} that there is a $Q\in\mathcal{Q}_L$ such that $Q$ is $K/2$-stacked above $\theta-\pi/2$. Since $\theta\in P_0(Q')$ for some $Q'\in\mathcal{G}_N$, it follows that $Q'\succ Q$, and thus by Lemma \ref{key1}, $Q'$ is $K/4$-stacked above $\theta-\pi/2$.
Recalling the definition of stacked disks, we conclude that
\begin{equation*}
P_0(\bigcup_{Q\in\mathcal{G}_N}Q )\setminus E_L\subseteq\{\theta:Mf_{N,\theta-\pi/2}(0)\geq \frac{K}{4}\}.
\end{equation*}
Lemma \ref{HL-vis} says that
\begin{equation*}
\{\theta:Mf_{N,\theta-\pi/2}(0)\geq \frac{K}{4}\}\subseteq\{\theta:Mg_N(\theta)\gtrsim K\}.
\end{equation*}
By the Hardy-Littlewood inequality,
\begin{equation*}
 \begin{split}
|P_0(\bigcup_{Q\in\mathcal{G}_N}Q)\setminus E_L| & \leq |\{\theta:Mg_N(\theta)\gtrsim K\}|\\
&\leq\frac1{K}\int g_N(\theta) d\theta\\
&\sim\frac1{K},
 \end{split}
\end{equation*}
which proves \eqref{HLBoundGoodWords}.
\end{proof}
%
%


\section{Some properties of self-similar sets}\label{unrect}
\subsection{Discrete unrectifiability of self-similar sets}\label{unrectA}
\begin{theorem}\label{discrete-unrect}
Let $0<\alpha\leq 1$. Let $\mathcal{J}$ be a self-similar set satisfying the Open Set Condition (see Definition \ref{defnOfSelfSim}) with $\dim(\mathcal{J})=\alpha$. Assume further that $\mathcal{J}$ is not contained in a line.

(a) Assume that $\alpha=1$. Then for $\delta>0$, $\mathcal J^\delta$ is equivalent to a $(\kappa,C_0,\delta)$-unrectifiable one-set for some $C_0$, $\kappa$ depending only on $\mathcal{J}$ but not on $\delta$.

(b) Assume that $0<\alpha<1$. Then for $\delta>0$, $\mathcal{J}^\delta$ is equivalent to a $(\alpha,C_0,\delta)$--set that is unconcentrated on lines, with $C_0$ depending only on $\mathcal{J}$ but not on $\delta$.

(c) In part (b), more is true. Assume that $0<\alpha<1$ and let $\varphi\colon \RR^2\to\RR^2$ be a diffeomorphism. Then for $\delta>0$, $\varphi(\mathcal{J}^\delta)$ is equivalent to a $(\alpha,C_0,\delta)$--set that is unconcentrated on lines, with $C_0$ depending only on $\mathcal{J}$ and $\varphi$, but not on $\delta$.
\end{theorem}
\begin{remark}
If $\alpha=1$ and $\varphi\colon\RR^2\to\RR^2$ is a diffeomorphism, then Theorem \ref{discrete-unrect}a and Proposition \ref{diffeo} implies that $\varphi(\mathcal{J}^\delta)$ is equivalent to a $(\kappa,C_0,\delta)$-unrectifiable one-set for some $C_0$ depending only on $\mathcal{J}$ and $\varphi,$ and some $\kappa$ depending only on $\mathcal{J}$. Since Proposition \ref{diffeo} doesn't hold for $(\alpha,C_0,\delta)$--sets that are unconcentrated on lines, we cannot apply it to obtain Theorem \ref{discrete-unrect}c from Theorem \ref{discrete-unrect}b. This is why we must prove Theorem \ref{discrete-unrect}c directly.
\end{remark}

Before proving Theorem \ref{discrete-unrect}, we will first need several preliminary lemmas.
\begin{lemma}\label{ballEstimateSelfSimilarScaling}
Let $\mathcal{J}$ be a self-similar set satisfying the Open Set Condition with $\dim(\mathcal{J})=\alpha$, where $0<\alpha\leq 1$. Then
\begin{equation}
C^{-1}\delta^{2-\alpha}|\leq |\mathcal{J}^{\delta}| \leq C\delta^{2-\alpha}.
\end{equation}
where the constant $C$ is independent of $\delta$.
\end{lemma}
\begin{proof}
We repeat the argument from \cite[Theorem 5.7]{mattila-book}. Since $\mathcal{J}$ is self-similar and satisfies the Open Set Condition, $\mathcal{J}$ is Ahlfors-David regular (see e.g.~\cite{Hutchinson}). In particular, for all $x\in\mathcal{J}$ and all balls $B$ of radius $r\leq\diam(\mathcal J)$, we have
\begin{equation}\label{H1Estimate}
 \mathcal{H}^\alpha(B\cap\mathcal{J})\sim  r^\alpha,
\end{equation}
where the implicit constants are independent of $r$ and the choice of ball. Choose a maximal $\delta$-separated set $A\subset \mathcal{J}$, then
\begin{equation}\label{H1e1}
\bigcup_{a\in A} B(a,\delta)\subseteq \mathcal{J}^{\delta}\subseteq \bigcup_{a\in A}B(a,2\delta)
\end{equation}
and each collection of balls is finitely overlapping. In particular, this implies that $|A|\sim \delta^{-2}|\mathcal{J}^{\delta}|$. By (\ref{H1Estimate}), the second inclusion in (\ref{H1e1}) implies that
\begin{equation}
\begin{split}
1&\lesssim \mathcal{H}^\alpha(\mathcal{J})\\
&\lesssim \sum_{a\in A} \mathcal{H}^\alpha(\mathcal{J}\cap B(a,2\delta))\\
&\lesssim |A| \, \delta^\alpha\\
&\lesssim (\delta^{-2}|\mathcal{J}^{\delta}|)\delta^\alpha.
\end{split}
\end{equation}
Thus $|\mathcal{J}^{\delta}|\gtrsim\delta^{2-\alpha}$. Similarly, the first inclusion in (\ref{H1e1}) implies that $|\mathcal{J}^{\delta}|\lesssim\delta^{2-\alpha}$.
\end{proof}

\begin{lemma}\label{ballEstimateSelfSimilar}
Let $\mathcal{J}$ be a self-similar set satisfying the Open Set Condition with $\dim(\mathcal{J})=\alpha$, where $0<\alpha\leq 1$. Then for every ball $B=B(x,r)$ of radius $r\geq \delta$, we have the bound
\begin{equation}
|\mathcal{J}^{\delta}\cap B|\leq C_0 r^\alpha|\mathcal{J}^{\delta}|,
\end{equation}
where the constant $C_0$ is independent of $r$ and the choice of ball.
\end{lemma}
\begin{proof}
Let $A\subset \mathcal{J}$ be a maximal $\delta$-separated set as in the proof of Lemma \ref{ballEstimateSelfSimilarScaling}, and let $A'=A\cap B$.
Then $|A'|\gtrsim  \delta^{-2}|\mathcal{J}^{\delta}\cap B|$, and the balls $B(a,\delta)$ with $a\in A'$ are finitely overlapping.
By (\ref{H1Estimate}), we have
\begin{equation}
\begin{split}
r^\alpha & \gtrsim \mathcal{H}^\alpha(\mathcal{J}\cap B(x,r+\delta))\\
&\gtrsim \sum_{a\in A'} \mathcal{H}^\alpha (\mathcal{J}\cap B(a,\delta))\\
&\gtrsim |A'|\,\delta^\alpha  \\
&\gtrsim \delta^{-2+\alpha} |\mathcal{J}^{\delta}\cap B|.
\end{split}
\end{equation}
Combining this with Lemma \ref{ballEstimateSelfSimilarScaling}, we have
\begin{equation*}
|\mathcal{J}^{\delta}\cap B|\lesssim r^\alpha\delta^{2-\alpha} \lesssim r^\alpha|\mathcal{J}^{\delta}|.\qedhere
\end{equation*}
\end{proof}

\begin{definition}
Recall that $W_n := \{1,\ldots,s\}^n$ and $W=\bigcup_{n=0}^\infty W_n$. For $w=(w_1,\ldots,w_n)\in W_n$, let $T_w := T_{w_n}\circ\cdots\circ T_1$, and let $\lambda_w := \prod_{j=1}^n \lambda_{w_j}.$
\end{definition}
\begin{lemma}\label{reFormulationLemma}
Let $\mathcal J$ be a self-similar set generated by similitudes $T_1,\dots,T_s$, satisfying the Open Set Condition, and not contained in a line. Then there exists a constant $c>0$, a number $s^\prime\geq s$, and a collection of similitudes $T_1^\prime,\ldots,T_{s^\prime}^\prime$ such that
\begin{equation*}
\mathcal{J} = \bigcup_{i=1}^{s^\prime} T^\prime_i(\mathcal J).
\end{equation*}
Furthermore, the similitudes satisfy the Open Set Condition, and they have the property that for any line $\ell\subset\RR^2$, there exists an index $1\leq i\leq s^\prime$ such that $T^\prime_i(\mathcal J)$ is disjoint from $\ell^{2c}$.
\end{lemma}

\begin{proof}
The set $\mathcal{J}$ is the closure of the set of the fixed points of the simiitudes $T_w$, $w\in W$ \cite[Theorem 3 (v)]{Hutchinson}.  Since $\mathcal{J}$ is not contained in a line, there are words $w_i^*\in W_{n_i}$, $i=1,2,3$, and a constant $\epsilon>0$ such that the fixed points $z_i^*$ of $T_{w_i^*}$ cannot all be contained in the $\epsilon$--neighborhood of a line.
Let
\begin{equation*}
w_1^{**} = \overbrace{w_1^* \stackrel{\phantom{|}}{\ldots} w_1^*}^{n_2n_3M\ \textrm{times}},\quad
w_2^{**} = \overbrace{w_2^* \stackrel{\phantom{|}}{\ldots} w_2^*}^{n_1n_3M\ \textrm{times}},\quad
w_3^{**} = \overbrace{w_3^* \stackrel{\phantom{|}}{\ldots} w_3^*}^{n_1n_2M\ \textrm{times}}.
\end{equation*}
Choose $M$ sufficiently large so that for $i=1,2,3$ we have $T_{w}(\mathcal J)\subset B(z_i^*,\epsilon/4)$ whenever $w$ is equal to, or a descendant of, $w_i^{**}$.

Let $M^*=n_1n_2n_3M$, and relabel the collection $\{T_w:\ w\in W_{M^*}\}$ as $\{T^\prime_i:\ i=1,...,s'\}$ with $s'=s^{M^*}$. It is clear that this extended family of similitudes generates the same self-similar set $\mathcal{J}$. Furthermore, $w_i^*\in W_{M^*}$ for $i=1,2,3$, and given any line $\ell$, at least one of the sets $T_{w_i^*}(\mathcal J)$ is disjoint from $\ell^{\epsilon/2}$. Thus the conclusion of the lemma holds with $c=\epsilon/4$.
\end{proof}
\begin{lemma}\label{stripEstimate}
Let $\mathcal J$ be a self-similar set satisfying the Open Set Condition and not contained in a line.
There exist constants $\kappa>0$ and $C$ such that for any line $\ell\subset\RR^2$, and any $\delta,\rho$ with $0<\delta<\rho\leq 1,$
\begin{equation}
|\mathcal J^\delta \cap \ell^\rho| \leq C \rho^\kappa|\mathcal J^\delta|.
\end{equation}
$\kappa$ and $C$ are independent of $\delta$ and $\rho$.
\end{lemma}
\begin{proof}

The proof is based on iterating Lemma \ref{reFormulationLemma}.
To simplify notation, we shall assume that the similitudes $T^\prime_1,\ldots,T^\prime_{s^\prime}$ from Lemma \ref{reFormulationLemma} were the original ones. For $j=1,\dots,s$, let $W_1^{(j)}\subset W_1$ be the set of one-letter words in the alphabet $\{1,\dots,s\}\setminus\{j\}$.

We may assume that $\rho<c/2$, since otherwise the result is immediate if $C$ is sufficiently large.
Then there is an index $i$ such that $(T_i(\mathcal J))^{c}$ is disjoint from $\ell^{c}$. Let $W_1^*=W_1^{(i)}$, then $|W_1^*|=s-1$. Let $c_1=\sum_{i=1}^s \lambda_i^\alpha -\lambda_{\min}^{\alpha} =1-\lambda_{\min}^\alpha<1$. Then $\sum_{w\in W_1^{(j)}}\lambda_w^\alpha\leq c_1$,
and in particular
$$\sum_{w\in W_1^*}\lambda_w^\alpha\leq c_1.$$
We now iterate the procedure. For each $w\in W_1^*$, the set $T_w(\mathcal{J})$ is a similar copy of $\mathcal{J}$, rescaled by the factor $\lambda_w\geq\lambda_{\min}$. By a rescaling of Lemma \ref{reFormulationLemma}, there is a letter $k(w)\in \{1,\ldots,s\}$ such that $(T_{k(w)}\circ T_w(\mathcal J))^{c\lambda_{\min}}$ is disjoint from $\ell^{c\lambda_{\min}}$. Let $W_2^*$ be the set of all words of the form $ww'$ with $w\in W_1^*$ and $w'\in W_1^{(k(w))}$. Then $|W_2^*|=(s-1)^2$. Furthermore, we have
$$
\sum_{w^*\in W_2^*}\lambda_{w^*}^\alpha=\sum_{w\in W_1^*}\lambda_w^\alpha\sum_{w^\prime\in W_1^{( k(w))}}\lambda_{w^\prime}^\alpha
\leq \sum_{w\in W_1^*}\lambda_w^\alpha c_1\leq c_1^2.
$$
Continuing in this manner for $m=3,4,\dots$, we find sets $W_m^*\subset W_m$ such that $|W_m^*|=(s-1)^m$
and  $(T_{w}(\mathcal J))^{c\lambda_{\min}^{m-1}}$ is disjoint from $\ell^{c\lambda_{\min}^{m-1}}$ for $w\in W_m\backslash W_m^*$. Moreover, we have
\begin{equation}\label{e-ew}
\sum_{w\in W_m^*}\lambda_w^\alpha \leq c_1^m.
\end{equation}
We halt the procedure when $c\lambda_{\min}^m\leq \rho\leq c\lambda_{\min}^{m-1}$, so that $m\sim
\frac{\log\rho}{\log\lambda_{\min}}$. At that stage, we have
$$
\mathcal{J}^\delta\cap \ell^\rho
\subseteq \mathcal{J}^{c\lambda_{\min}^{m-1}} \cap \ell^\rho
\subseteq \bigcup_{w\in W_m^*} (T_w(\mathcal{J}))^{c\lambda_{\min}^{m-1}} \cap\ell^\rho.
$$
Each set $ (T_w(\mathcal{J}))^{c\lambda_{\min}^{m-1}} $ is contained in a ball $B_w$ of radius
$\lesssim \lambda_w$, with the implicit constant independent of $m$ and $w$. If necessary, we may increase this constant by a factor $\sim 1$ so that each $B_w$ has radius greater than $\delta$. By Lemma \ref{ballEstimateSelfSimilar} and (\ref{e-ew}), we have
\begin{align*}
|\mathcal{J}^\delta\cap \ell^\rho &|\leq \sum_{w\in W_m^*} |\mathcal{J}^\delta\cap B_w|\\
&\lesssim \sum_{w\in W_m^*} \lambda_w^\alpha |\mathcal{J}^\delta|\\
&\lesssim c_1^m  |\mathcal{J}^\delta|\\
&\sim c_1^{(\log\rho)/(\log \lambda_{\min})} |\mathcal{J}^\delta|\\
&=\rho^{\kappa}  |\mathcal{J}^\delta|
\end{align*}
with $\kappa=\frac{\log c_1}{\log\lambda_{\min}}>0$.
\end{proof}

\begin{lemma}\label{rectEstimateLem}
Let $\mathcal J$ be a self-similar set of dimension $\alpha$ with $0<\alpha\leq 1$, satisfying the Open Set Condition and not contained in a line. Then there exist some $0<\kappa\leq\alpha$ and a constant $C$ so that
\begin{equation}
|\mathcal J^\delta \cap R| \leq C r_1^{\kappa}r_2^{\alpha-\kappa}|\mathcal J|
\end{equation}
whenever $R$ is a rectangle of dimensions $\delta\leq r_1\leq r_2$.
\end{lemma}
\begin{proof}
We may assume that $r_1\leq C_1^{-1} r_2$ for some large $C_1$, since otherwise the lemma follows trivially from Lemma \ref{stripEstimate}.
Choose a ball $B_0:=B(a, 2r_2)$ so that $R\subset B_0$.
For each $w\in W$ such that $T_w(\mathcal J)\cap B_0\neq\emptyset$, let $\tilde w$ be the shortest word such that $w$ is a child of $\tilde w$ and  $\lambda_{\tilde w} \leq r_2.$ (Note that we then also have
$\lambda_{\tilde{w}}\geq r_2\lambda_{\min} \sim r_2$.)
Let $\mathcal W$ be the set of such maximal words. Then $(T_w(\mathcal{J}))^\delta\subset B(a,C_2r_2)$ for some $C_2\sim 1$. Note also that if $w,w'\in\mathcal{W}$, then $w$ cannot be a descendant of $w'$.

Since $\mathcal{J}$ satisfies the Open Set Condition, we have
$H^\alpha (T_w(\mathcal J)\cap T_{w'}(\mathcal J))=0$ if $w,w'\in\mathcal{W}$ and $w\neq w'$ (see \cite[Section 4.13]{mattila-book}).
Furthermore,
$$H^\alpha (T_w(\mathcal J))\sim \lambda_w^\alpha (\mathcal{J})\sim r_2^\alpha.$$
It follows from this and (\ref{H1Estimate}) that
$$
r_2^\alpha\gtrsim H^\alpha (\mathcal J)\cap B(a,C_2r_2))
\gtrsim \sum_{w\in\mathcal{W}}H^\alpha (T_w(\mathcal J))
\gtrsim r_2^\alpha |\mathcal{W}|,
$$
so that $|\mathcal{W}|\sim 1$.

For each word $w\in\mathcal W$, scale the set $(T_{w}(\mathcal J))^\delta$ by a factor of $\lambda_{w}^{-1}\sim r_2^{-1}$, obtaining a homothetic copy of $\mathcal J^{\delta/\lambda_{w}}$. The image of $R$ under the same  scaling is a rectangle $R^\prime(w)$ of dimensions $\lambda_w^{-1}r_2 \times \lambda_w^{-1}{r_1}$.
By Lemma \ref{stripEstimate}, we have
\begin{equation}
|\mathcal J^{\delta/\lambda_w}\cap R^\prime(w)|\leq C (r_1/\lambda_{w})^\kappa|\mathcal J^{\delta/\lambda_{w}}|
\end{equation}
Undoing the scaling and using Lemma \ref{ballEstimateSelfSimilarScaling}, we obtain
\begin{equation}\label{contribOneWord}
\begin{split}
|(T_{w}(\mathcal J))^\delta \cap R| & \lesssim r_2^\alpha  (r_1/\lambda_{w})^\kappa |\mathcal J^\delta|\\
&\lesssim r_1^{\kappa}r_2^{\alpha-\kappa}|\mathcal J^\delta|.
\end{split}
\end{equation}
Since $\mathcal J^\delta \cap R\subseteq \bigcup_{w\in\mathcal{W}}((T_w(\mathcal J))^\delta\cap R)$ and
$|\mathcal{W}|\sim 1$, the lemma follows.
\end{proof}

\begin{proof}[Proof of Theorem \ref{discrete-unrect}]
We will first prove parts (a) and (b) of the theorem. By Lemmas \ref{ballEstimateSelfSimilarScaling} and \ref{ballEstimateSelfSimilar}, we have
$C^{-1}\delta^{2-\alpha}\leq |\mathcal{J}^{\delta}| \leq C\delta^{2-\alpha}$, and furthermore
$\mathcal{J}^\delta$ obeys (\ref{unconcentratedBallEstimate}) (note that for $r\leq\delta$, the last estimate is trivial). The bound (\ref{unconcentratedLineEstimate}) follows from Lemma \ref{stripEstimate}. Moreover,
if $\alpha=1$, the estimate (\ref{defnOfOneDimUnrecEqn}) follows from Lemma \ref{rectEstimateLem}.

We will now prove part (c). Let $\varphi\colon\RR^2\to\RR^2$ be a diffeomorphism, and let $\ell\subset\RR^2$ be a line.  We need to show that for $C$ sufficiently large (depending only on $\mathcal J$ and $\varphi$),
\begin{equation}\label{diffeoLineConc}
|\varphi(\mathcal{J}^\delta\cap\ell^{1/C})| \leq |\varphi(\mathcal{J}^\delta)|/10.
\end{equation}
We will show that for every $C_1>0$, there is a constant $C_2$ so that
\begin{equation}\label{diffeoLineCondition}
|\mathcal{J}^\delta\cap(\varphi^{-1}(\ell))^{1/C_2}|\leq C_1^{-1}|\mathcal{J}^\delta|.
\end{equation}
Since $\varphi$ has Jacobian $\sim 1$ on the convex hull of $\mathcal{J},$ \eqref{diffeoLineCondition} will imply \eqref{diffeoLineConc}.

By Lemma \ref{ballEstimateSelfSimilarScaling}, $\mathcal{J}^{C_2^{-1/2}}$ can be coved by $O(C_2^{\alpha/2})$ balls of radius $C_2^{1/2}$. This implies that $\mathcal{J}^{\delta}$ can be covered by $O(C_2^{\alpha/2})$ balls of radius $C_2^{1/2}$. Let $B$ be one of these balls. Then $B\cap \mathcal{J}^\delta\cap(\varphi^{-1}(\ell))^{1/C_2}$ is contained within $O(1)$ rectangles of dimensions $C_1^{-1/2}\times C_1^{-1}$. By Lemma \ref{rectEstimateLem}, we have
\begin{equation*}
B\cap \mathcal{J}^\delta\cap(\varphi^{-1}(\ell))^{1/C_2}\lesssim (C_2^{-\kappa})(C_2^{-1/2})^{\alpha-\kappa}|\mathcal{J}^\delta|
\end{equation*}
Summing the contribution from all $O(C_2^{\alpha/2})$ balls, we conclude that
\begin{equation*}
\begin{split}
\mathcal{J}^\delta\cap(\varphi^{-1}(\ell))^{1/C_2}&\lesssim C_2^{\alpha/2} (C_2^{-\kappa})(C_2^{-1/2})^{\alpha-\kappa}|\mathcal{J}^\delta|\\
&\lesssim C_2^{-\kappa/2}|\mathcal{J}^\delta|.
\end{split}
\end{equation*}
Thus if we select $C_2$ sufficiently large compared to $C_1$, we obtain \eqref{diffeoLineCondition}.

The only remaining point is that $\mathcal{J}^\delta$ might not be a union of finitely overlapping $\delta$-balls.
Choose a maximal $\delta$-separated set $A\subset \mathcal{J}$ as in the proof of Lemma \ref{ballEstimateSelfSimilarScaling}, and let $\mathcal{A}=\bigcup_{a\in A} B(a,\delta)$. This is a finitely overlapping collection of balls. By (\ref{H1e1}),
we have
$\mathcal{A}\subseteq \mathcal{J}^{\delta}$, and conversely, $\mathcal{J}^\delta$ can be covered by finitely many translates of $\mathcal{A}$. Thus $\mathcal J^\delta$ is equivalent to $\mathcal{A}$. It follows that all of the above estimates hold with $\mathcal{J}^\delta$ replaced by $\mathcal{A}$. In particular, when $\alpha=1$ then $\mathcal{A}$ is a $(\kappa,C_0,\delta)$-unrectifiable one-set, and for $0<\alpha<1$, $\mathcal{A}$ is a $(\alpha,C_0,\delta)$--set that is unconcentrated on lines.
\end{proof}


\subsection{Self-similar sets have large projection in every direction}
The next lemma is used in the proof of Proposition \ref{hochman-prop}.

\begin{lemma}\label{easy}
Let $\mathcal{J}$ be a self-similar set not contained in a line. Then there is an $\alpha>0$ such that $\dim\pi_\theta(\mathcal{J})\geq\alpha$ for all $\theta\in[0,2\pi]$.
\end{lemma}

\begin{proof}
We will use the notation from the proof of Theorem \ref{discrete-unrect}.

First, fix $\theta$. Since $\mathcal{J}$ is not contained in a line, for all $n$ large enough we may find words $w_1,w_2\in W_n$ such that $\pi_\theta(Q_{w_1})$ and $\pi_\theta(Q_{w_2})$ are disjoint. Note further that if $n'>n$, then the same is true with $w_1,w_2$ replaced by any pair of their respective descendants $w'_1,w'_2\in W_{n'}$.

It is clear from the construction above that the same $n$, with the same words $w_1$ and $w_2$, works also for $\theta'$ in a small enough neighbourhood $U(\theta)$ of $\theta$. By compactness and the argument above, we may find a value of $n$ that works for all $\theta$. Let $\lambda=\lambda_{\min}^n.$ The interval $\pi_\theta(Q_w)$ has length at least $c\lambda$ for all $w\in W_n$.
Let $\alpha=\frac{\log\lambda}{\log(1/2)}>0$.

It is then easy to see that for each $\theta$, the set $\pi_\theta(\mathcal{J})$ contains a (not necessarily self-similar) Cantor set of dimension at least $\alpha$, obtained by iterating the construction above. Specifically, for each fixed $\theta$, we have two disjoint intervals $I_1=\pi_\theta(Q_{w_1})$ and $I_2=\pi_\theta(Q_{w_2})$ of length at least $\lambda$ contained in $\pi_\theta(\mathcal{J}_n)$. Continuing by induction, the sets $Q_{w_i}\cap \mathcal{J}_n$ contain a (possibly rotated) self-similar copy of $\mathcal{J}_n$, so that each of the sets
$I_i\cap \pi_\theta( \mathcal{J}_{2n})$ contains at least two disjoint intervals $I_{i,1}$ and $I_{i,2}$ of length at least $\lambda^2$ that are obtained by projecting discs of $\mathcal{J}_{2n}$, and so on. This proves the claim.
\end{proof}


\bibliographystyle{amsplain}

\noindent{\sc Bond, {\L}aba: Department of Mathematics, University of British Columbia, Vancouver,
B.C. V6T 1Z2, Canada}

\smallskip

\noindent{\it  bondmatt@math.ubc.ca, ilaba@math.ubc.ca}

\medskip

\noindent{\sc Zahl: Department of Mathematics, MIT, Cambridge MA, 02139, USA}

\smallskip

\noindent{\it  jzahl@mit.edu}

\end{document}